\documentclass[aps,pre,twocolumn,floatfix]{revtex4}

\usepackage[normalem]{ulem}
\usepackage{latexsym}
\usepackage[usenames]{color}
\usepackage{amsthm}
\usepackage{hyperref}
\usepackage{braket}

\usepackage{amsmath}    
\usepackage{amsfonts}  
\usepackage{amssymb}
\usepackage{graphicx}
\usepackage{slashed}
\usepackage{fouridx}
\usepackage{tikz}
\usetikzlibrary{calc}
\theoremstyle{plain} 
\theoremstyle{plain} 
\theoremstyle{plain} \newtheorem{theor}{Theorem}[section] 
\theoremstyle{plain} 
\theoremstyle{plain}  
\theoremstyle{remark} 
\theoremstyle{plain} 
\theoremstyle{remark}

\newcommand\CROSS[1]{%
  \hbox{%
    \vbox{
      \hrule
      \kern2.5pt
      \hbox{$#1$\,\strut}
    }%
  \vrule
  }\mskip\thickmuskip
}

\tikzset{
solid node/.style={circle,draw,inner sep=1.5,fill=black},
hollow node/.style={circle,draw,inner sep=1.5}
}

\newlength{\arrowsize}  
\pgfarrowsdeclare{biggertip}{biggertip}{  
  \setlength{\arrowsize}{0.5pt}  
  \addtolength{\arrowsize}{.5\pgflinewidth}  
  \pgfarrowsrightextend{0}  
  \pgfarrowsleftextend{-5\arrowsize}  
}{  
  \setlength{\arrowsize}{0.5pt}  
  \addtolength{\arrowsize}{.5\pgflinewidth}  
  \pgfpathmoveto{\pgfpoint{-5\arrowsize}{4\arrowsize}}  
  \pgfpathlineto{\pgfpointorigin}  
  \pgfpathlineto{\pgfpoint{-5\arrowsize}{-4\arrowsize}}  
  \pgfusepathqstroke  
}

\begin{document}
\title{A constructive theory of shape}

\author{Vladimir Garc\'{\i}a-Morales}

\affiliation{Departament de F\'{\i}sica de la Terra i Termodin\`amica, Universitat de Val\`encia, E-46100 Burjassot, Spain}
\email{garmovla@uv.es}

\begin{abstract}
\noindent We formulate a theory of shape valid for objects of arbitrary dimension whose contours are path connected. We apply this theory to the design and modeling of viable trajectories of complex dynamical systems. Infinite families of qualitatively similar shapes are constructed giving as input a finite ordered set of characteristic points (landmarks) and the value of a continuous parameter $\kappa \in (0,\infty)$.  We prove that all shapes belonging to the same family are located within the convex hull of the landmarks. The theory is constructive in the sense that it provides a systematic means to build a mathematical model for any shape taken from the physical world. We illustrate this with a variety of examples: (chaotic) time series, plane curves, space filling curves, knots and strange attractors. ~\\

 \end{abstract}
\maketitle

\section{Introduction}

A fact that has captivated the human mind since ancient times is the diversity of shapes found in nature \cite{Ball}. It is, however, only in the last half-century that mathematical theories have been developed to study shape more rigorously and systematically, from topological \cite{Borsuk, Mardesic} and statistical \cite{Kendall1, Kendall2, Dryden} points of view. The great scientific interest of shape is that it is a unifying concept \cite{Leyton}. Besides being a property of all living systems \cite{Thompson}, shape is of crucial importance in architecture \cite{Alexander} and art \cite{Leyton, Leypain, Leyarch}. Physical explanations for the origin of shape in non-equilibrium systems have been proposed \cite{Bejan} that are of great interest in engineering and optimization problems.  Quantitative analysis of the shapes of biological organs is often necessary in agronomy, medicine, genetics, ecology, and taxonomy \cite{Bookstein, Iwata2, Costa}. 

Shapes and images are essentially sets \cite{Aubin3}. Therefore, the analysis, processing, evolution, regulation and control of shape require the consideration of set-valued maps. Viability theory \cite{Aubin1, Aubin2} and the study and reconstruction of invariant sets \cite{Blanchini1, Blanchini2, Nagumo,Takens} provide pathways to address this problem. Here, we present an alternative approach linking it to a nonlinear dynamical systems perspective. Although our theory is general in scope, it is able to provide specific mathematical models of shapes and their viable trajectories. We say that a trajectory is viable in a set $S$ if, for an initial condition in $S$, the trajectory remains in $S$ in any future time.

Spline approximation, Fourier analysis \cite{Kuhl} and wavelet transforms  \cite{Barache, Osowski} are popular methods for analyzing shape. The two general approaches to shape are: region-based (the region in the image corresponding to the analyzed object is considered); and boundary-based (shape is characterized in terms of its silhouette) \cite{Barache, Pavlidis}. Although we shall focus in this article in the second approach, representing shape as a parametric contour and endowing this representation with `nice' properties, the first approach is also possible with help of the concept of $\kappa$-families of scenes introduced in this article. Besides generality of the method, a most desirable property that is sought is \emph{ease}: once the contour of the shape is traced and a \emph{finite} sequence of points in this contour (the landmarks) is selected, a mathematical model for the specific shape and for an infinite family of qualitatively related shapes can be immediately written down. There is no need of calculating any coefficients. This requirement is, clearly, not satisfied by any of the above methods. It is this property of the theory here presented what makes it \emph{constructive}. The theory is intended to be used as a mathematical tool for scientists to provide quick and handy mathematical models of complex shapes. The theory is also parsimonious because a minimal amount of information suffices to generate an infinite family of qualitatively related shapes. Most remarkably, all shapes belonging to the family are located within the set constituted by the convex hull of the landmarks. Thus, if the landmarks are meant to be points on the trajectory of a complex dynamical system (possibly subject to perturbations) an infinite number of viable trajectories for the dynamical system can be obtained, all contained in the convex hull of the landmarks.

Our theory of shape is based in nonlinear $\mathcal{B}_{\kappa}$-embeddings, a mathematical structure that has been recently introduced by the author \cite{homotopon, JPHYSCOMPLEX} and applied to the problem of finding all the roots of a polynomial in the complex plane \cite{homotopon}. The theory is able to represent any path-connected shape and to resolve details of the shape at different scales as a continuous smoothing parameter $\kappa$ is varied. This work is motivated by the following general problem, to which we give a general solution: \emph{For a given finite sequence of data (points) in $\mathbb{R}^{n}$ find a mathematical model for the shape of the data that: a) can fit the given data to any degree of accuracy; b) the interpolation between the data is a non-oscillating, non-piecewise, infinitely differentiable function of its variable (`time') argument  $t$; c)  relates each individual shape to closely related shapes (deformations) obtained by varying the continuous parameter $\kappa$; d) places any $\kappa$-deformed shape within the convex hull of the original data}.

The outline of this article is as follows. In Section \ref{theory} we present our definitions of shape and $\kappa$-family of shapes preceded by all the concepts necessary to understand the definitions. Nonlinear embeddings are introduced and the properties of interest here are worked out to make the theory self-contained. In Section \ref{convex} we prove that the $\kappa$-family of shapes is contained within the convex hull of the landmarks. Then, in Section \ref{dynapp} this general result is applied to different classes of dynamical systems and it is established how viable trajectories can be constructed for them. In Section \ref{examples} we give examples of application of the theory to curves in 1D, periodic shapes (waveforms), curves in 2D (including fractal space-filling curves \cite{Hilbert}), curves in 3D (e.g. knots \cite{Kaufman} and strange attractors). The impatient reader can directly jump to this section going back to Section \ref{theory} when needed. In Section \ref{conclusions} we present some conclusions and discussion of the theory and sketch some directions for possible future work.

\section{Nonlinear embeddings and the definition of shape} \label{theory}

We consider $N$ points in $D$-dimensional Euclidean space extracted from a contour. We shall call these points \emph{characteristic points} or \emph{landmarks} of the shape. We choose a labeling for these landmarks so that they are ordered one after the other in succession starting with point $0$ up to point $N-1$. These landmarks can be obtained directly from the boundary by judicious, intuitive or automated choice. Neither a chain code \cite{Freeman, Bribiesca} nor  tracing the whole contour is needed (as it is the case with Fourier methods as \cite{Kuhl}).  The labeling of the landmarks is important because, as we shall see, it is related to time. In the case of closed contours any cyclic permutation of the landmarks is equivalent. The points are thus, specified by vectors $\mathbf{r}_0, \mathbf{r}_1,\ldots, \mathbf{r}_{N-1}$. By introducing a Kronecker delta convolution, we can rewrite the $n$-th point ($0\le n \le N-1$) as
\begin{equation}
\mathbf{r}_{n}=\sum_{j=0}^{N-1}\mathbf{r}_{j}\delta_{jn}
\end{equation}
where $\delta_{nj}=1$ if $n=j$ and $\delta_{nj}=0$ otherwise, is the Kronecker delta. We note that, indeed, we can write, equivalently,
\begin{equation}
\mathbf{r}_{n}=\frac{\sum_{j=0}^{N-1}\mathbf{r}_{j}\delta_{jn}}{\sum_{j=0}^{N-1}\delta_{jn}}, \label{conv1}
\end{equation}
since $\sum_{j=0}^{N-1}\delta_{jn}=1$ because $0\le n \le N-1$. The construction of a nonlinear embedding begins by noting that
the Kronecker delta admits the following simple representation \cite{JPHYSCOMPLEX}
\begin{equation}
\delta_{nj}=\mathcal{B}\left(n-j,\frac{1}{2}\right) \label{d2}
\end{equation} 
where 
\begin{eqnarray}
\mathcal{B}(x,y)&\equiv&\frac{1}{2}\left(\frac{x+y}{|x+y|}-\frac{x-y}{|x-y|}\right) \nonumber \\
&=&{\begin{cases} \text{sign}(y)&{\text{if }}|x| < |y|\\ \text{sign}(y)/2 &{\text{if }}|x|=|y|, y\ne 0 \\0& {\text{otherwise}} \end{cases}} \label{d1}
\end{eqnarray} 
is the $\mathcal{B}$-function \cite{VGM1, VGM2, JPHYSCOMPLEX}, with 
$x, y \in \mathbb{R}$. We can thus use Eq. (\ref{d2}) to replace the Kronecker deltas entering in Eq. (\ref{conv1}). We obtain
 \begin{equation}
\mathbf{r}_{n}=\frac{\sum_{j=0}^{N-1}\mathbf{r}_{j}\mathcal{B}\left(n-j,\frac{1}{2}\right)}{\sum_{j=0}^{N-1}\mathcal{B}\left(n-j,\frac{1}{2}\right)}, \label{conv2}
\end{equation}
We can now `fuzzify' this expression by means of the approach sketched in \cite{fuzzypap, JPHYSCOMPLEX}. First, we replace the discrete variable $n \in \mathbb{Z}$ by a continuous time variable $t \in \mathbb{R}$. We then define
\begin{equation}
\mathbf{r}(t):=\frac{\sum_{j=0}^{N-1}\mathbf{r}_{j}\mathcal{B}\left(t-j,\frac{1}{2}\right)}{\sum_{j=0}^{N-1}\mathcal{B}\left(t-j,\frac{1}{2}\right)}, \label{conv4}
\end{equation}
Obviously, at times $t=n+\delta$, where $-1/2<\delta<1/2$
\begin{equation}
\mathbf{r}(n+\delta)=\mathbf{r}_{n} \label{equivs}
\end{equation}
The fuzzification approach is completed by replacing all $\mathcal{B}$-functions in Eq. (\ref{conv4}) by the $\mathcal{B}_{\kappa}$-function \cite{fuzzypap} given by
\begin{equation}
\mathcal{B}_{\kappa}(x,y):= \frac{1}{2}\left[ 
\tanh\left(\frac{x+y}{\kappa} \right)-\tanh\left(\frac{x-y}{\kappa} \right)
\right] \label{bkappa}
\end{equation}
This is the so-called \emph{replacement mode I} in \cite{JPHYSCOMPLEX}. We then define the nonlinear embedding
\begin{equation}
\mathbf{r}_{\kappa}(t):=\frac{\sum_{j=0}^{N-1}\mathbf{r}_{j}\mathcal{B}_{\kappa}\left(t-j,\frac{1}{2}\right)}{\sum_{j=0}^{N-1}\mathcal{B}_{\kappa}\left(t-j,\frac{1}{2}\right)}. \label{conv5}
\end{equation}
We observe, that in the limit $\kappa \to 0$, Eq. (\ref{conv5}) becomes Eq. (\ref{conv4}) 
\begin{equation}
\lim_{\kappa \to 0} \mathbf{r}_{\kappa}(t)=\mathbf{r}(t) \label{limes}
\end{equation}
Because of Eqs. (\ref{equivs}) and (\ref{limes}), Eq. (\ref{conv5}) is, in turn, equal to Eq. (\ref{conv1}) in the limit $\kappa \to 0$ when $t\in [0,N-1]$ is an integer.
When $\kappa \to \infty$ we have, since $\mathcal{B}_{\kappa}(x,y) \to y/\kappa$ \cite{JPHYSA, homotopon}
\begin{equation}
\lim_{\kappa \to \infty} \mathbf{r}_{\kappa}(t)=\frac{\sum_{j=0}^{N-1}\mathbf{r}_{j}\frac{1}{2\kappa}}{\sum_{j=0}^{N-1}\frac{1}{2\kappa}}=\frac{1}{N}\sum_{j=0}^{N-1}\mathbf{r}_{j}=\mathbf{r}^{*}
\end{equation}
where we have introduced the \emph{centroid} (`center of masses') $\mathbf{r}^{*}$ of the distribution of landmarks.

We note that, since \cite{JPHYSA, JPHYSCOMPLEX}
\begin{equation}
\sum_{j=0}^{N-1}\mathcal{B}_{\kappa}\left(t-j,\frac{1}{2}\right)=\mathcal{B}_{\kappa}\left(t-\frac{N-1}{2},\frac{N}{2}\right),
\end{equation}
we can write Eq. (\ref{conv5}) as
\begin{equation}
\mathbf{r}_{\kappa}(t)=\sum_{j=0}^{N-1}\mathbf{r}_{j}\frac{\mathcal{B}_{\kappa}\left(t-j,\frac{1}{2}\right)}{\mathcal{B}_{\kappa}\left(t-\frac{N-1}{2},\frac{N}{2}\right)} \quad t\in [0,N-1], t\in \mathbb{R}. \label{type1}
\end{equation}
This equation defines \emph{open} shapes. We shall also consider \emph{closed} shapes. In the letter, although they are still specified by a finite set of landmarks, these return periodically as $t$ is varied from $-\infty$ to $\infty$. We have
\begin{equation}
\mathbf{r}_{n+kN}=\mathbf{r}_{n}, \qquad \forall k \in \mathbb{Z},
\end{equation}
so that landmark $\mathbf{r}_{0}$ is the successor of landmark $\mathbf{r}_{N-1}$. In this case, instead of Eq. (\ref{conv5}), we shall consider the following expression
\begin{equation}
\mathbf{r}_{\kappa}(t):=\frac{\sum_{j=0}^{N-1}\mathbf{r}_{j}\Pi_{\kappa}\left(t-j,\frac{1}{2}; N \right)}{\sum_{j=0}^{N-1}\Pi_{\kappa}\left(t-j,\frac{1}{2}; N\right)}, \quad t\in [0,\infty), t\in \mathbb{R}, \label{type2}
\end{equation}
where
\begin{equation}
\Pi_{\kappa}\left(x,y; T \right)=\mathcal{B}_{\kappa}\left(\sin \frac{\pi x}{T}, \sin \frac{\pi y}{T} \right)
\end{equation}
is the periodic $\mathcal{B}_{\kappa}$-function introduced in \cite{fuzzypap}. From this equation, we note that
\begin{equation}
\mathbf{r}_{\kappa}(t+kN)=\mathbf{r}_{\kappa}(t), \qquad \forall k \in \mathbb{Z}.
\end{equation}
When $\kappa \to \infty$, Eq. (\ref{type2}) reduces to
\begin{equation}
\lim_{\kappa \to \infty} \mathbf{r}_{\kappa}(t)=\frac{1}{N}\sum_{j=0}^{N-1}\mathbf{r}_{j}=\mathbf{r}^{*},
\end{equation}
i.e. the shape collapses to centroid. Eqs. (\ref{type1}) and (\ref{type2}) are the main equations of our approach. 

We now formally define shape. Let $\mathbf{r}_{0}$, $\mathbf{r}_{1}$, $\ldots$, $\mathbf{r}_{N-1}$ be a sequence of $N$ points in $D$-dimensional Euclidean space $\mathbb{R}^{D}$. These points may be repeated in the collection. We call these points \emph{characteristic points of the shape} or, equivalently, \emph{landmarks}. Over these landmarks, we define a $\kappa$\emph{-family of open shapes} by Eq. (\ref{type1}) and  a $\kappa$\emph{-family of closed shapes} by Eq. (\ref{type2}). For each specific value of $\kappa \in (0, \infty)$ we say that the latter equations define an open (resp. closed) shape. We say that a $\kappa$-family of shapes is equal to another if both have exactly the same landmarks in the same order.

We define the $\epsilon$-shape of a given $\kappa$-family as the resulting shape $\mathbf{r}_{\epsilon}(t)$ when $\kappa=\epsilon$ is nonzero but vanishingly small. In all examples in Section \ref{examples} the $\epsilon$-shape is visually indistinguishable of the shapes obtained for $\kappa=0.01$. We say that the $\epsilon$-shape is a faithful representation of the shape to be modelled if, to sufficient degree of accuracy, fits the original shape (note that from the original shape we are only taking a finite set of $N$ landmarks). All $\epsilon$-shapes in the examples are faithful. A faithful representation can always be obtained by taking a sufficient number $N$ of landmarks that are representative of its contour.

Eqs. (\ref{type1}) and (\ref{type2}) constitute \emph{parametric equations} of the shapes and it is useful to see parameter $t$ as time. Although in this work we shall focus on landmarks in $\mathbb{R}^{D}$, other vector spaces are possible as well. Any landmark $n$, $n=0, 1,\ldots, N-1$, can  be recovered from Eqs. (\ref{type1}) and (\ref{type2}) by putting $t=n$ in these equations and taking the limit $\kappa \to 0$
\begin{equation}
\mathbf{r}_{n}=\lim_{\kappa \to 0} \mathbf{r}_{\kappa}(n)
\end{equation} 
The process of obtaining a particular shape begins by specifying a finite suitable set of landmarks. We discuss a variety of examples on how to construct mathematical models for $\kappa$-families of shapes.
In general the landmarks can be specified in a wide variety of methods. For example, they can be obtained: (1) from experimentally measured data or model; (2) randomly (according to a probability distribution) or at will;  (3) by reading them directly from a pre-existing figure; (4) recursively.

The $\kappa$-family constitutes a mathematical model in which shapes are parametrized by means of a `time' real parameter $t$ (a larger number of parameters is possible) and a smoothing real parameter $\kappa$. The shapes within the family are controlled by means of $\kappa$ and are located between the  interpolated landmarks (the $\epsilon$-shape) and the centroid of the landmarks.

It is possible to hierarchically embed families of shapes in more complex structures that we shall call \emph{scenes}. A scene is any disconnected set of shapes and can be expressed as a direct product of them. Let $\mathbf{r}_{\kappa_0}(t,0)$, $\mathbf{r}_{\kappa_1}(t,2)$, $\ldots$, $\mathbf{r}_{\kappa_M-1}(t,M-1)$ be a different $\kappa$-family of shapes given either by Eq. (\ref{type1}) or by Eq. (\ref{type2}) each. We can construct a type-I $\eta$-family of scenes $\mathbf{S}_{\eta}(t)$ as
\begin{eqnarray}
\mathbf{S}_{\eta}(t,q)&=&\sum_{j=0}^{M-1}\mathbf{r}_{\kappa_{j}}(t,j)\frac{\mathcal{B}_{\eta}\left(q-j,\frac{1}{2}\right)}{\mathcal{B}_{\eta}\left(q-\frac{M-1}{2},\frac{M}{2}\right)} \label{Stype1} \\ && t\in [0,N-1], q\in [0,M-1], t,q\in \mathbb{R}. \nonumber 
\end{eqnarray}
and a type-II $\eta$-family of scenes as
\begin{eqnarray}
\mathbf{S}_{\eta}(t,q)&=&\frac{\sum_{j=0}^{M-1}\mathbf{r}_{\kappa_{j}}(t,j)\Pi_{\eta}\left(q-j,\frac{1}{2}; M \right)}{\sum_{j=0}^{M-1}\Pi_{\eta}\left(t-j,\frac{1}{2}; M\right)} \label{Stype2} \\ && t,q \in [0,\infty), t,q\in \mathbb{R}. \nonumber 
\end{eqnarray}
We see that shapes play within a scene an analogous role to landmarks within a shape. In the limit $\eta \to 0$ all shapes within the scene are disconnected. In the limit $\eta \to \infty$ the shapes within the scene mix with each other and collapse to the average shape
\begin{equation}
\lim_{\eta \to \infty}\mathbf{S}_{\eta}(t,q)=\frac{1}{M}\sum_{j=0}^{M-1}\mathbf{r}_{\kappa_{j}}(t,j)
\end{equation}
Each shape $m$ within a $\kappa$-family of scenes can be recovered from Eq. (\ref{Stype1}) and (\ref{Stype2}) by putting $q=m \in [0, M-1]$ and taking the limit $\eta \to 0$. 
 \begin{equation}
\lim_{\eta \to 0}\mathbf{S}_{\eta}(t,n)=\mathbf{r}_{\kappa_{n}}(t,n)
\end{equation}

Scenes can be hierarchically embedded in increasingly complex structures in a similar way as landmarks are embedded in shapes and the latter are embedded in scenes. Every element in the complex structure, landmark, shape, or scene, can thus be recovered as sketched above. We find that our theory thus satisfies Leyton's criteria for a successful theory of shape: 1) \emph{maximization of transfer} and 2) \emph{maximization of recoverability} \cite{Leyton}. The first criterion is satisfied because complex structures in our theory can always be seen as the result of the transfer of simpler structures since these are embedded in the former. The second criterion is also satisfied because all elements that are being transferred can be recovered out of simple operations (taking limits) on the superior structures. Leyton's generative theory of shape is abstract and makes use of group theoretical notions (the wreath product being at its core) \cite{Leyton}. We believe than our  theory of shape is simpler because, being general, it is also specific in the following sense: a mathematical model can always be fully worked out explicitly and systematically for any shape taken as example (and this is why we call it a \emph{constructive} theory).

\section{Convexity and transformations} \label{convex}

The $\kappa$-families of shapes obtained from Eqs. (\ref{type1}) and (\ref{type2}) have several important mathematical properties that we now discuss. We have the following result.

\begin{theor} \label{theconvex} Let \emph{conv}$_{\mathbf{r}}$ denote the convex hull of the points $\mathbf{r}_0$, $\mathbf{r}_1$, $\ldots$, $\mathbf{r}_{N-1}$ in $D$-dimensional Euclidean space $\mathbb{R}^{D}$. Then: i) the points $\mathbf{r}_{\kappa}(t)$ in any open shape of a type-1 $\kappa$-family of shapes given by Eq. (\ref{type1}) satisfy $\mathbf{r}_{\kappa}(t) \in$ \emph{conv}$_{\mathbf{r}}$, $\forall \kappa \in (0,\infty)$ and $\forall t\in [0,N-1], t\in \mathbb{R}$; ii) the points $\mathbf{r}_{\kappa}(t)$ in any closed shape of a type-2 $\kappa$-family of shapes given by Eq. (\ref{type2}) satisfy $\mathbf{r}_{\kappa}(t) \in$ \emph{conv}$_{\mathbf{r}}$, $\forall \kappa \in (0,\infty)$ and $\forall t\in [0,\infty), t\in \mathbb{R}$.
\end{theor}

\begin{proof} Since $\text{conv}_{\mathbf{r}}$ is a convex set and $\mathbf{r}_{j} \in  \text{conv}_{\mathbf{r}}$, for $j=0, \ldots, N-1$ we have that any affine combination
\begin{equation}
\mathbf{r}'=\sum _{j=0}^{N-1}\lambda _{j}\mathbf{r}_{j}  \label{affine}
\end{equation}
satisfies $\mathbf{r}' \in  \text{conv}_{\mathbf{r}}$ for any non-negative real numbers $\lambda_0$, $\ldots$, $\lambda_{N-1}$ such that $\lambda_0+\lambda_1+\ldots \lambda_{N-1}=1$ \cite{Rockafellar}. If we, therefore, take
\begin{equation}
\lambda_{j}=\frac{\mathcal{B}_{\kappa}\left(t-j,\frac{1}{2}\right)}{\mathcal{B}_{\kappa}\left(t-\frac{N-1}{2},\frac{N}{2}\right)}=\frac{\mathcal{B}_{\kappa}\left(t-j,\frac{1}{2}\right)}{\sum_{j=0}^{N-1}\mathcal{B}_{\kappa}\left(t-j,\frac{1}{2}\right)} \label{cho1}
\end{equation}
we have both $0\le \lambda_{j} \le 1$, and $\sum_{j=0}^{N-1}\lambda_{j}=1$. By replacing Eq. (\ref{cho1}) in Eq. (\ref{affine}) we obtain Eq. (\ref{type1}) and i) follows. Similarly, if we take
\begin{equation}
\lambda_{j}=\frac{\Pi_{\kappa}\left(t-j,\frac{1}{2}; N\right)}{\sum_{j=0}^{N-1}\Pi_{\kappa}\left(t-j,\frac{1}{2}; N\right)} \label{cho2}
\end{equation}
we, again, have $0\le \lambda_{j} \le 1$, and $\sum_{j=0}^{N-1}\lambda_{j}=1$. By replacing Eq. (\ref{cho2}) in Eq. (\ref{affine}) we obtain Eq. (\ref{type2}) and ii) follows.
\end{proof}

The importance of the above theorem lies in that all deformed shapes within the same $\kappa$-family are easily located in Euclidean space $\mathbb{R}^{D}$: they are all  found within the convex hull of the landmarks. As we shall see in Sec. \ref{dynapp}, Theorem \ref{theconvex} finds an important application in finding viable trajectories of complex dynamical systems and in designing invariant sets. 

The $\kappa$-families of shapes in Eqs. (\ref{type1}) and (\ref{type2}) are \emph{nonlinear} in $t$ and $\kappa$ but are \emph{linear} functions (combinations) of the landmarks. This has the following important implication: if $\mathbf{T}$ is any transformation (matrix) sending each landmark $\mathbf{r}_{j}$ to $\mathbf{T}\cdot \mathbf{r}_{j}$ (with $\cdot$ denoting the inner product), then the $\kappa$ family of shapes transforms as $\mathbf{r}_{\kappa}(t) \to \mathbf{T}\cdot \mathbf{r}_{\kappa}(t)$. Thus, if, for example, $\mathbf{T}$ is an isometry transformation, the shapes $\mathbf{r}_{\kappa}(t)$ and $\mathbf{T}\cdot \mathbf{r}_{\kappa}(t)$ for a fixed $\kappa$ are congruent.

\section{Application to dynamical systems} \label{dynapp}

A smooth dynamical system
\begin{equation}
\dot{\mathbf{y}}=\mathbf{F}(\mathbf{y}) \label{smooth}
\end{equation}
can be discretized in time as
\begin{equation}
\mathbf{y}(t+\delta)\approx \mathbf{y}(t)+\delta \mathbf{F}(\mathbf{y}(t))
\end{equation}
In general, we have, from this latter equation
\begin{equation}
\mathbf{y}(t+(n+1)\delta)\approx\mathbf{y}(t+n\delta)+\delta \mathbf{F}(\mathbf{y}(t+n\delta)) \label{evoldyns}
\end{equation}
Therefore, starting from $t=0$ and ending in time $t=N\delta$, we can obtain $N$ suitable landmarks for the trajectory of the dynamical system up to resolution $\delta$  as $\left(\mathbf{r}_0, \mathbf{r}_{1}, \mathbf{r}_{2}, \ldots, \mathbf{r}_{N-1} \right):=\left(\mathbf{y}(0), \mathbf{y}(\delta), \mathbf{y}(2\delta), \ldots, \mathbf{y}((N-1)\delta) \right)$  
and the trajectory can be embedded in a $\kappa$-family of open shapes by means of Eq. (\ref{type1}) where the $\mathbf{r}_{j}$'s are iteratively obtained from Eq. (\ref{evoldyns}) as
\begin{equation}
\mathbf{r}_{j}=\mathbf{r}_{j-1}+\delta \mathbf{F}(\mathbf{r}_{j-1})  \label{evolsmooth}
\end{equation}
starting from an initial condition $\mathbf{r}_{0}$. 

We can also obtain periodic, closed shapes, over these same  $N$ landmarks from Eq. (\ref{type2}). Any trajectory of the smooth dynamical system can thus be embedded in a $\kappa$-family of shapes by taking $\delta $ sufficiently small so that the landmarks obtained constitute a faithful representation of the trajectory up to precision $\delta$. As an example, in Sec. \ref{Lorenz} we construct $\kappa$-families of shapes for the Lorenz attractor.

If the dynamical system is a discrete map, instead of Eq. (\ref{smooth}) we have that the dynamics is governed at discrete times $t=j$ by 
\begin{equation}
\mathbf{y}_{j}=\mathbf{G}(\mathbf{y}_{j-1}) \label{discmap}
\end{equation}
A set of $N$ iterations starting from an initial condition $\mathbf{y}_0$ can be made with this map so that $N$ landmarks $\mathbf{r}_{j}=\mathbf{y}_j$ are generated, which exactly coincide with the trajectory of the discrete map. The trajectory jumps discontinuously at discrete times, but can be embedded in a \emph{continuous} trajectory by means of an appropriate $\kappa$-family of shapes (type I) or (type II) by using Eqs. (\ref{type1}) and (\ref{type2}) above.  In Sec. \ref{logimap} we show, as example, open $\kappa$-families of shapes for the logistic map.

By using the concept of $\eta$-family of scenes defined by Eqs. (\ref{Stype1}) and (\ref{Stype2}), it is possible to generalize the above costructions to spatially extended systems as partial differential equations, coupled map lattices and cellular automata. 

The extension of the above embeddings to non-autonomous dynamical systems is straigthforward. Of major interest in those complex systems (in the presence of controls and perturbations) is the notion of positive invariance \cite{Blanchini1, Blanchini2, Nagumo}. A set $S$ is said to be positively invariant with respect to a dynamical system if any trajectory $\mathbf{y}(t)$ of the dynamical system that is well defined for any $t>0$ and that starts with initial condition $\mathbf{y}_0 \in S$ satisfies $\mathbf{y}(t)\in S$ for $t>0$. A set that contains a stable limit cycle (closed periodic orbit) or a stable fixed point that attract all initial conditions in $S$, is clearly positively invariant. 

A broader notion of positive invariance is crucial in so-called viability theory \cite{Aubin1, Aubin2}. Many complex natural and artificial systems, organizations, and networks do not evolve deterministically nor stochastically \cite{Aubin1,Aubin2}. These systems are approached in a general way considering the trajectories $\mathbf{r}(t)$ (system states) to be constrained to belong to a set $S$ instead of being governed by systems of  differential or stochastic equations. Indeed, the main requirement is positive invariance for the trajectories defined analytically or experimentally, so that one has
\begin{eqnarray}
\mathbf{r}(t) &\in& S.  \label{viable}
\end{eqnarray}
An important consequence of Theorem \ref{theconvex} in this context is that if $S$ is convex then we can always augment such a complex system given by Eq. (\ref{viable}) with a control parameter $\kappa$ so that for all $\kappa > 0$ we have, as well,
\begin{eqnarray}
\mathbf{r}_{\kappa}(t) &\in& S  \label{viable2}
\end{eqnarray}
with $\mathbf{y}_{\kappa}(t)$ being given by Eq. (\ref{type1}) or Eq. (\ref{type2}).   Indeed, as a direct corollary of Theorem \ref{theconvex} applied to dynamical systems and positively invariant sets, we have the following.

\begin{theor} Let $S$ be any convex set that is positively invariant with respect to a dynamical system with globally defined trajectories $\mathbf{y}(t)\in S$ ($\forall t>0$) starting from initial conditions $\mathbf{y}_{0}\in S$.  Let $\mathbf{r}_{j}$, $j=0,\ldots, N-1$ be $N$ points taken on any of these trajectories. Then, all points in the curve $\mathbf{r}_{\kappa}(t)$ obtained from Eq. (\ref{type1}) satisfy $\mathbf{r}_{\kappa}(t) \in S$, $\forall \kappa >0$ and $\forall t \in [0, N-1]$.
\end{theor}

Any convex set $S$ includes as a subset the convex hull $\text{conv}_{\mathbf{r}}$ of any set of points $\{\mathbf{r}_{j}; j=0,\ldots, N-1\} \subset S$. Then, another direct application of Theorem \ref{theconvex} to closed trajectories of any dynamical system (continuous or discrete) is the following.

\begin{theor} Let $\mathbf{y}(t)=\mathbf{y}(t+T)$ denote a closed periodic orbit of a dynamical system with period $T$ and let $\mathbf{r}_{j}$, $j=0,\ldots, N-1$ be $N$ points taken on the 
orbit. Then $\mathbf{r}_{\kappa}(t)$ given by Eq. (\ref{type2}) is a closed curve with period $T$ and is contained in the convex hull  \emph{conv}$_{\mathbf{r}}$ of the points taken on the trajectory of the dynamical system, i.e. $\mathbf{r}_{\kappa}(t) \in$ \emph{conv}$_{\mathbf{r}}$, $\forall \kappa >0$ and $\forall t>0$.
\end{theor}

\section{Examples} \label{examples}

\subsection{1D shapes: time series} \label{logimap}

In one dimension, the landmarks $\mathbf{r}_{0}$, $\mathbf{r}_{1}$, $\ldots$, $\mathbf{r}_{N-1}$ become just a collection of real scalars $X_0$, $X_1$, $\ldots$, $X_{N-1}$. It is straightforward to obtain from these landmarks and Eqs. (\ref{type1}) and (\ref{type2}) $\kappa$-families of shapes. For open shapes (non-periodic time series on a time window) we have
\begin{eqnarray}
r_{\kappa}(t)&=&\sum_{j=0}^{N-1}X_{j}\frac{\mathcal{B}_{\kappa}\left(t-j,\frac{1}{2}\right)}{\mathcal{B}_{\kappa}\left(t-\frac{N-1}{2},\frac{N}{2}\right)} \label{1Dtype1} \\ &&\qquad \qquad \qquad \qquad \qquad t\in [0,N-1], t\in \mathbb{R}, \nonumber
\end{eqnarray}
and for periodic time series
\begin{eqnarray}
r_{\kappa}(t)&=&\frac{\sum_{j=0}^{N-1}X_{j}\Pi_{\kappa}\left(t-j,\frac{1}{2}; N\right)}{\sum_{j=0}^{N-1}\Pi_{\kappa}\left(t-j,\frac{1}{2}; N\right)} \qquad t\in \mathbb{R}. \qquad \label{1Dtype2}
\end{eqnarray}

\begin{figure*}
\includegraphics[width=0.85 \textwidth]{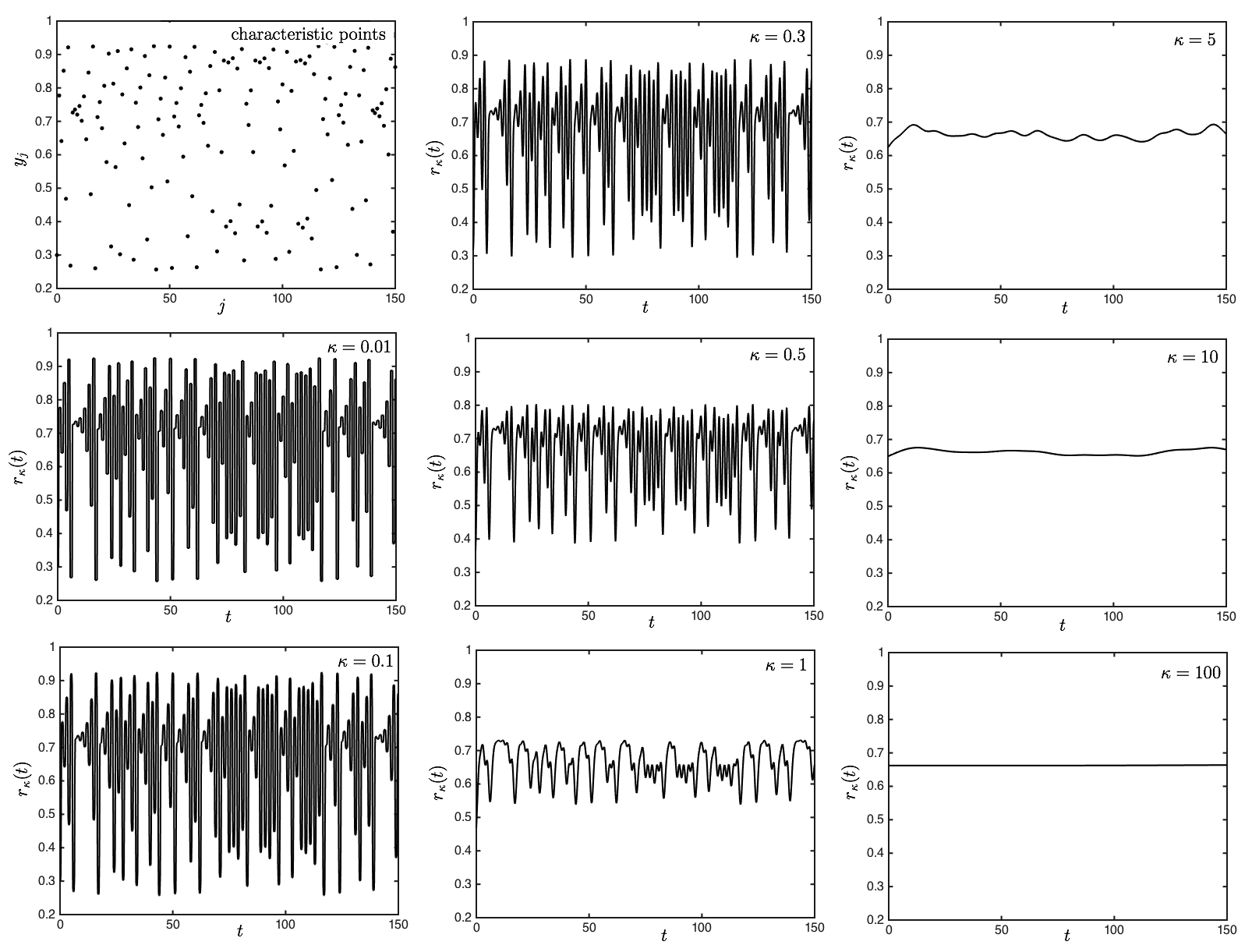}
\caption{\scriptsize{Shapes $r_{\kappa}(t)$, calculated from Eq. (\ref{logikappa}) for the landmarks (characteristic points) $y_j$ shown in the top left panel for the values of the continuum parameter $\kappa$ indicated on the panels. The landmarks are calculated by performing 149 iterations of the logistic map $y_{j}=\mu y_{j-1} (1-y_{j-1})$, with $y_{0}=0.3$ and $\mu=3.5$.}} \label{f11a}
\end{figure*}
\begin{figure*}
\includegraphics[width=0.8 \textwidth]{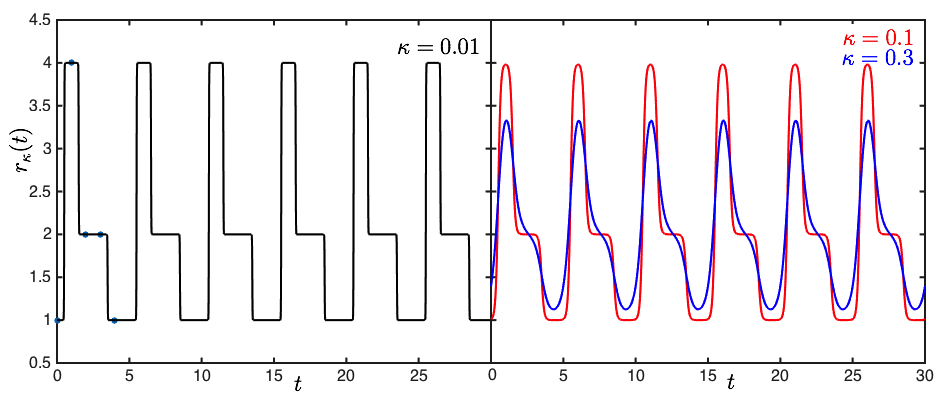}
\caption{\scriptsize{Periodic shapes obtained from the landmarks $(X_0,X_1,\ldots, X_{4})=(1,4,2,2,1)$ and Eq. (\ref{1Dtype2}) for the values of $\kappa$ indicated on the panels. In the left panel, the landmarks are also shown.}} \label{f11b}
\end{figure*}

As an example we consider the logistic map as a model providing the landmarks. We put $X_{j}=y_j$ where $y_j$ is given by
\begin{equation}
y_{j}=\mu y_{j-1} (1-y_{j-1})
\end{equation}
starting from a given $y_0$. If we perform 149 iterations of this map, we obtain a sequence of 150 landmarks and we can directly construct a $\kappa$-family of open shapes from these landmarks, using Eq. (\ref{type1})
 \begin{eqnarray}
r_{\kappa}(t)&=&y_0 \frac{\mathcal{B}_{\kappa}\left(t,\frac{1}{2}\right)}{\mathcal{B}_{\kappa}\left(t-\frac{149}{2},75\right)} \nonumber \\
&&+\mu   \sum_{j=1}^{149}y_{j-1} (1-y_{j-1})\frac{\mathcal{B}_{\kappa}\left(t-j,\frac{1}{2}\right)}{\mathcal{B}_{\kappa}\left(t-\frac{149}{2},75\right)} \qquad \label{logikappa}\\
&& \qquad \qquad \qquad \qquad \qquad t\in [0,149], t\in \mathbb{R}. \nonumber
\end{eqnarray}
In Fig. \ref{f11a}, $r_{\kappa}(t)$ obtained from Eq. (\ref{logikappa}) is plotted for different values of $\kappa$, an the parameters $\mu=3.5$ and $y_{0}=0.3$ selected to be in the chaotic regime. The landmarks (characteristic points) are also shown. In every case, the shapes $r_{\kappa}(t)$ obtained are smooth and infinitely differentiable. For $\kappa$ sufficiently low, the shapes interpolate the landmarks through the $\epsilon$-shape with growing precision as $\kappa \to 0$. For $\kappa$ large, details within a scale $s=2\kappa$ are averaged out and the resulting shapes have broader variations (the curve being always smooth). As $\kappa \to \infty$, the shape collapses on the average value of the landmarks for every $t$.

A possible use of the shape theory in 1D above sketched is the synthesis of sound. Complex periodic waveforms can be simply designed by giving the landmarks as input (these can also be read from recorded waveforms). No Fourier decomposition is necessary to select the right harmonics. Then, by using Eq. (\ref{type2}), an infinite family of periodic signals with the same frequency can be constructed from the input landmarks.

An example of an invented periodic signal is shown in Fig. \ref{f11b}. By choosing the five landmarks $(X_0,X_1,\ldots, X_{4})=(1,4,2,2,1)$ we can calculate any shape within the $\kappa$-family by replacing them in Eq. (\ref{1Dtype2}) and putting $N=5$. In Fig. \ref{f11b} the shapes with $\kappa=0.01, 0.1$ and $0.3$ are shown together with the landmarks. We observe that five points together with the parameter $\kappa$ is enough to specify an infinite family of qualitatively related waveforms. Note that, because for $\kappa =0.01$ the shapes tend to change more abruptly between the different landmarks, a huge number of Fourier modes would be needed to reproduce such periodic signals, the Gibbs phenomenon also playing an important role. No Gibbs or Runge phenomena are present in the interpolation provided by the $\epsilon$-shape in our theory because there are no oscillations of the $\epsilon$-shape between the landmark points under interpolation. We remark that this interpolation is \emph{not} piecewise: the shapes are all infinitely differentiable functions of both $t$ and $\kappa$.

\subsection{2D shapes: planar curves}

In two dimensions we have the vector $\mathbf{r}_{\kappa}(t)=(x_{\kappa}(t), y_{\kappa}(t))$ that yields a plane curve in parametric form and Eq. (\ref{type1}) for open shapes reduces to
\begin{eqnarray}
x_{\kappa}(t)&=&\sum_{j=0}^{N-1}X_{j}\frac{\mathcal{B}_{\kappa}\left(t-j,\frac{1}{2}\right)}{\mathcal{B}_{\kappa}\left(t-\frac{N-1}{2},\frac{N}{2}\right)} \label{2Dtype1x} \\
y_{\kappa}(t)&=&\sum_{j=0}^{N-1}Y_{j}\frac{\mathcal{B}_{\kappa}\left(t-j,\frac{1}{2}\right)}{\mathcal{B}_{\kappa}\left(t-\frac{N-1}{2},\frac{N}{2}\right)}  \label{2Dtype1y} \\ 
&& \qquad \qquad \qquad \qquad t\in [0,N-1], t\in \mathbb{R} \nonumber
\end{eqnarray}
where $(X_j, Y_j)=\mathbf{r}_{j}$ specifies the $j$-th landmark. For closed shapes in the plane Eq. (\ref{type2}) is equivalent to
\begin{eqnarray}
x_{\kappa}(t)&=&\frac{\sum_{j=0}^{N-1}X_{j}\Pi_{\kappa}\left(t-j,\frac{1}{2}; N\right)}{\sum_{j=0}^{N-1}\Pi_{\kappa}\left(t-j,\frac{1}{2}; N\right)} \label{2Dtype2x} \\ 
y_{\kappa}(t)&=&\frac{\sum_{j=0}^{N-1}Y_{j}\Pi_{\kappa}\left(t-j,\frac{1}{2}; N\right)}{\sum_{j=0}^{N-1}\Pi_{\kappa}\left(t-j,\frac{1}{2}; N\right)}   \qquad t\in \mathbb{R} \quad \label{2Dtype2y}
\end{eqnarray}
where $(X_j, Y_j)=\mathbf{r}_{j}$ specifies the $j$-th landmark and we have $(x_{\kappa}(t), y_{\kappa}(t))=\mathbf{r}_{\kappa}(t)$, as before.

An example of a $\kappa$-family of open shapes described by Eqs. (\ref{2Dtype1x}) and (\ref{2Dtype1y}) is shown in Fig. \ref{signat} in which 50 landmarks are chosen randomly within the unit square from the uniform distribution. For $\kappa=0.01$, the shape interpolates among all the landmarks but, as $\kappa$ is increased, the shapes within the family tend to shrink to the centroid of the landmarks. The convex hull of the landmarks is shown in orange shadow and it is observed that for all $\kappa >0$, the shapes are contained within the convex hull, as predicted by Theorem \ref{theconvex}. If the shape is regarded as a trajectory of a complex dynamical system, an infinite number of viable trajectories in the same set can be generated by means of the above method, selecting appropriate subsets of landmarks.  As $\kappa$ is increased, the finest details are gradually averaged out and, as a result, the shapes are smoothened, but they still constitute different deformations that look qualitatively similar.

\begin{figure*}
\begin{center}
\includegraphics[width=1.0 \textwidth]{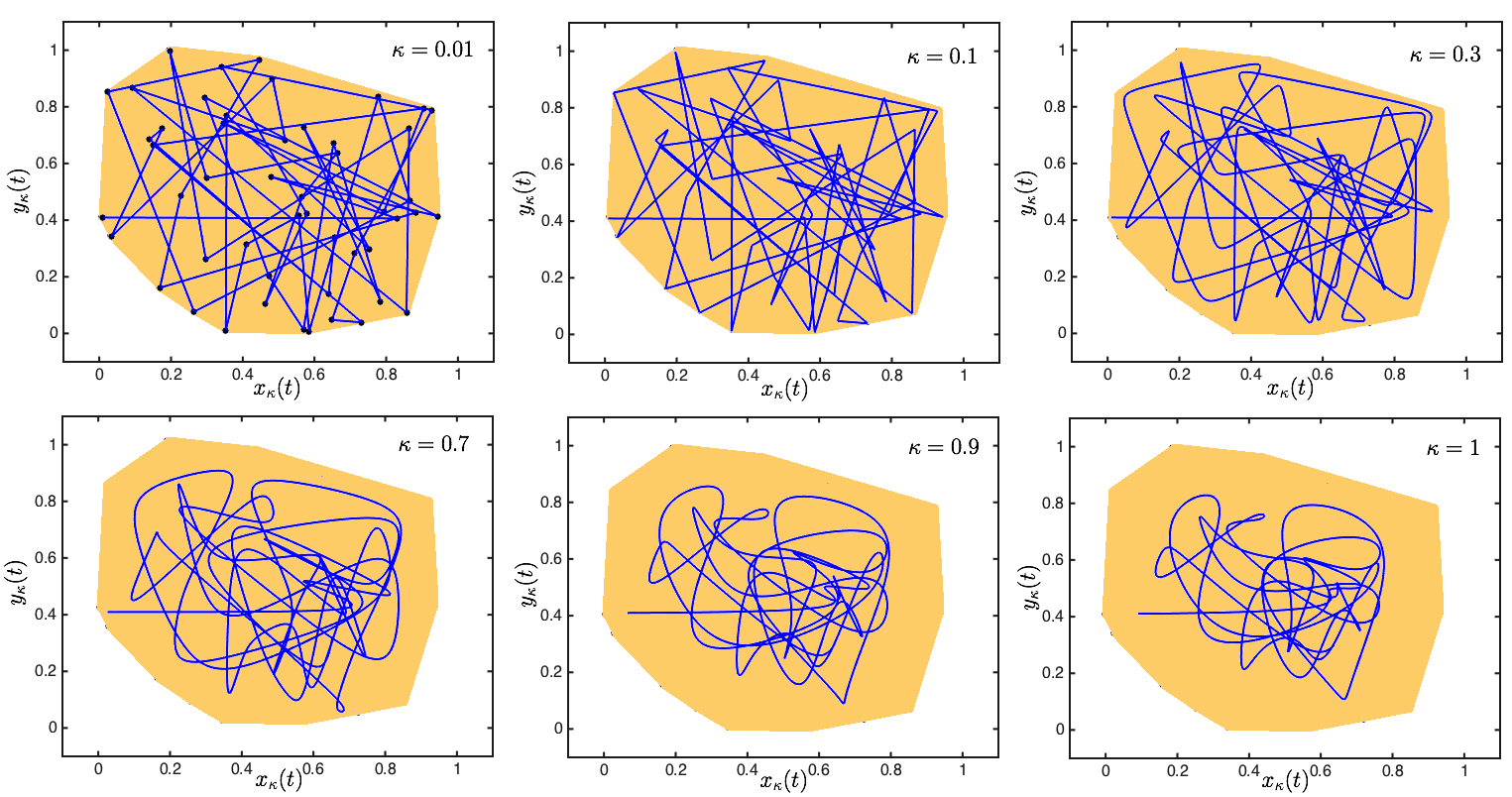}
\caption{\scriptsize{Shapes obtained from Eqs. (\ref{2Dtype1x}) and (\ref{2Dtype1y}) for 50 landmarks taken randomly in the unit square from the uniform distribution. Shown in orange shadow is the convex hull of the landmarks. Values of $\kappa$ are indicated on the panels and the landmarks are also shown in the top leftmost panel.}} \label{signat}
\end{center}
\end{figure*}

We now discuss closed shapes in 2D. John von Neumann's joke ``With four parameters I can fit an elephant, and with five I can make him wiggle his trunk'' \cite{Dyson} has motivated some  work on how experimental data on the plane can be fitted by parametric equations of planar curves \cite{Mayer, Wei, Piantadosi}. A most popular method is provided by expanding the $x(t)$ and $y(t)$ coordinates of a closed contour (in this work $x_{\kappa}(t)$ and $y_{\kappa}(t)$) in Fourier series \cite{Kuhl}
\begin{eqnarray}
x(t)&=&\sum_{n=0}^{\infty}\left[A_{n}\cos\left(\frac{2\pi nt}{T}\right)+B_{n}\sin\left(\frac{2\pi nt}{T}\right)\right] \qquad \label{f1} \\
y(t)&=&\sum_{n=0}^{\infty}\left[C_{n}\cos\left(\frac{2\pi nt}{T}\right)+D_{n}\sin\left(\frac{2\pi nt}{T}\right)\right] \qquad \label{f2}
\end{eqnarray}
where the coefficients $A_{n}$, $B_{n}$, $C_{n}$ and $D_{n}$ are related to $x(t)$ and $y(t)$ by the integrals
\begin{eqnarray}
A_n&=&\frac{2}{T}\int_{0}^{T}x(t)\cos\left(\frac{2\pi nt}{T}\right) \text{d}t, \\  B_n&=&\frac{2}{T}\int_{0}^{T}x(t)\sin\left(\frac{2\pi nt}{T}\right) \text{d}t, 
\end{eqnarray}
\begin{eqnarray}
C_n&=&\frac{2}{T}\int_{0}^{T}y(t)\cos\left(\frac{2\pi nt}{T}\right) \text{d}t, \\ D_n&=&\frac{2}{T}\int_{0}^{T}y(t)\sin\left(\frac{2\pi nt}{T}\right) \text{d}t.
\end{eqnarray}
A calculation of these integrals can be done by assuming $N$ distinct landmarks piecewisely joined by rectilinear segments by means of a chain code \cite{Kuhl}. Smooth shapes are then obtained by truncating the sums in Eqs. (\ref{f1}) and (\ref{f2}) \cite{Kuhl}. In this way, any closed shape in 2D can be described. This popular method \cite{McLellan2} has been applied to biological shapes in animals \cite{Bierbaum, Diaz, Ferson, Rohlf} and plants \cite{Furuta, Iwata, McLellan, Ohsawa, White} and has been extended to open shapes in 3D by means of the discrete cosine transform \cite{Zhou}. 

\begin{figure*}
\begin{center}
\includegraphics[width=0.9 \textwidth]{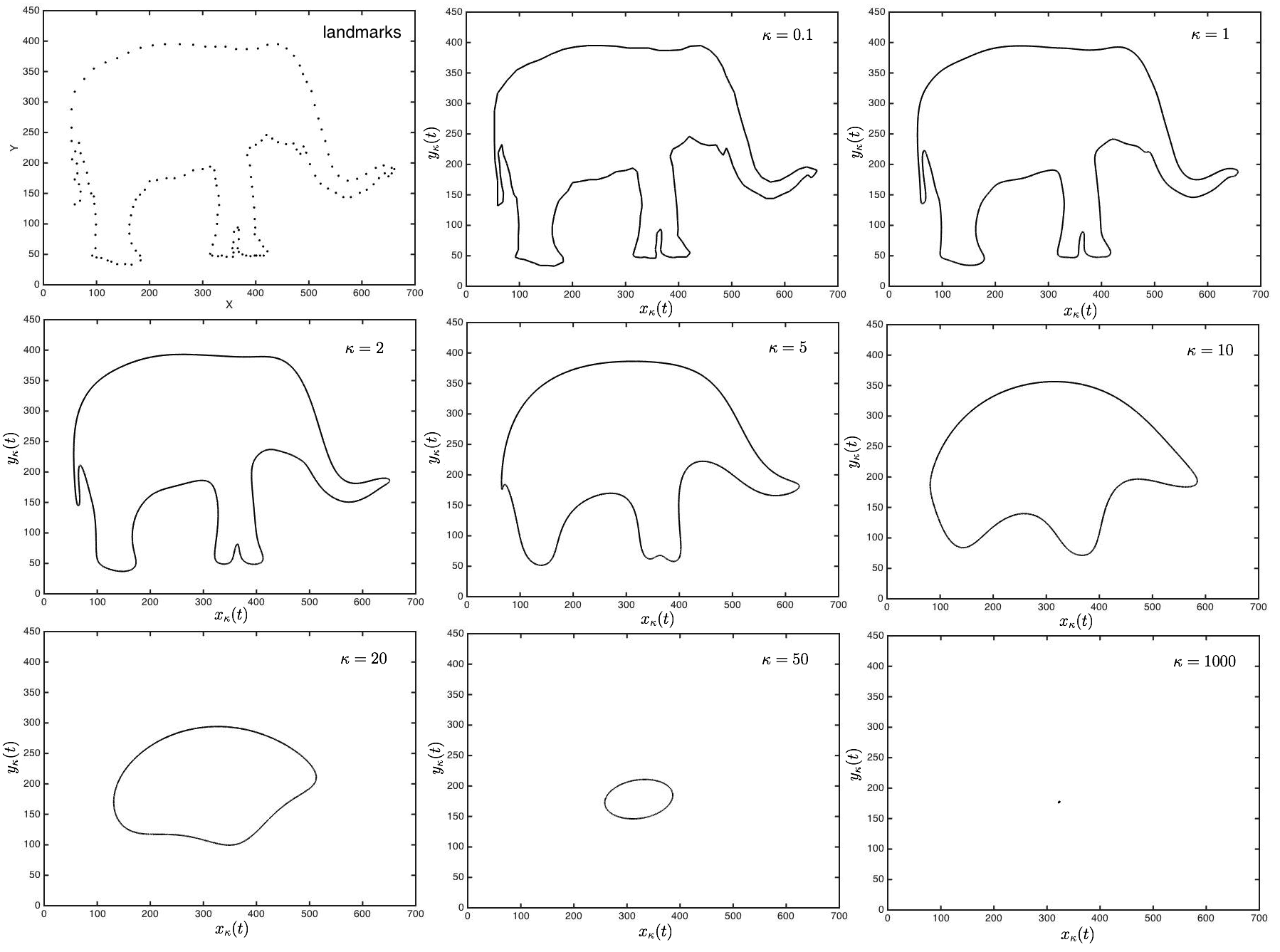}
\caption{\scriptsize{Plots of $y_{\kappa}(t)$ vs. $x_{\kappa}(t)$ obtained from Eqs. (\ref{2Dtype2x}) and (\ref{2Dtype2y})) for the landmarks shown in the top leftmost panel and the values of $\kappa$ in the panels. The $\kappa$-family of shapes contains the silhouette of the elephant and the centroid of the landmarks as its members (in the limits $\kappa \to 0$ and $\kappa \to \infty$ respectively) as well as an infinite family of shapes that qualitatively interpolate between both limiting cases.}} \label{eleph}
\end{center}
\end{figure*}

It is interesting to compare the Fourier method of \cite{Kuhl} with ours (embodied in Eqs. (\ref{2Dtype2x}) and (\ref{2Dtype2y})) since: \emph{1) our method does not require the evaluation of any coefficient since these are the landmarks themselves and $x_{\kappa}(t)$ and $y_{\kappa}(t)$ are linear functions of them; 2) the expansions in our method are always finite and equal to the finite number of landmarks also for shapes with vertices and sharp edges; 3) between landmarks the interpolation is always smooth and there are no oscillations; 4) no chain code \cite{Freeman, Bribiesca} is needed as in \cite{Kuhl} but merely the coordinates of the landmarks in successive order;  5)  the method is completely straightforward and an infinite family of qualitatively related shapes is obtained by continuously varying $\kappa$ between 0 and $\infty$}.

We note that only one parameter $\kappa$ is necessary to fit the data in our method, the $\epsilon$-shape being this fitting. It has been suggested that, indeed only one parameter is necessary to fit any arbitrary collection of data \cite{Piantadosi}. However, the fitting in \cite{Piantadosi} is extremely sensitive to the value of the parameter. The shapes obtained with our method have not this sensitivity: varying $\kappa$ continuously produces a set of qualitatively related shapes and if changes in $\kappa$ are small, the changes in shape also are. 

In Fig. \ref{eleph} we show a fitting of Eqs. (\ref{2Dtype2x}) and (\ref{2Dtype2y})) to the shape of an elephant. The 160 landmark points $(X_{j}, Y_{j})$, $j=0,\ldots 159$ are chosen in the contour of the elephant. Sufficiently small $\kappa$ values produce an excellent fit of the elephant. As $\kappa$ is increased, the closed shape gradually shrinks to the centroid of the landmarks. This centroid coincides with the zero Fourier mode in the Fourier expansion method of \cite{Kuhl}.  It is, however, to be noted that to apply the Fourier method first the shape must be traced with rectilinear segments \cite{Kuhl} connecting landmarks by a chain code \cite{Freeman, Bribiesca}. After that, Fourier coefficients need to be calculated up to a certain number and inserted in the Fourier expansions to yield the parametric equations of the shape. Our method only needs the landmarks directly extracted from the contour and these lead directly to the parametric equations of the shape thanks to the corresponding nonlinear $\mathcal{B}_{\kappa}$-embeddings.

In the following we shall write the landmarks $(X_0, Y_0)$, $(X_1, Y_1)$, $\ldots$, $(X_{N-1}, Y_{N-1})$ more compactly in a rectangular matrix as 
\begin{equation}
\left(
\begin{array}{cccc}
X_0  & X_1  & \ldots & X_{N-1}    \\
Y_0  & Y_1  & \ldots & Y_{N-1}
\end{array}
\right)
\end{equation}
The order of the landmarks and whether these are repeated or not is a defining property of the resulting $\kappa$-family of shapes. In Fig. \ref{polyg}A we show four shapes of the $\kappa$-family obtained from the landmarks
\begin{equation}
\left(
\begin{array}{ccc}
0  & 4  & 7    \\
0  & 3  & 1 \label{A}
\end{array}
\right)
\end{equation}
and for the values $\kappa=0.01$ (black curve), $\kappa=0.3$ (blue curve), $\kappa=0.5$ (red curve) and $\kappa=1$ (green curve).
By introducing the landmarks in Eqs. (\ref{2Dtype2x}) and (\ref{2Dtype2y})) we obtain the mathematical model in terms of parametric curves for the whole $\kappa$-family
\begin{eqnarray}
x_{\kappa}(t)&=&\frac{4\Pi_{\kappa}\left(t-1,\frac{1}{2}; 3\right)+7\Pi_{\kappa}\left(t-2,\frac{1}{2}; 3\right)}{\sum_{j=0}^{2}\Pi_{\kappa}\left(t-j,\frac{1}{2}; 3\right)} \quad \label{trianX}  \\
y_{\kappa}(t)&=&\frac{3\Pi_{\kappa}\left(t-1,\frac{1}{2}; 3\right)+\Pi_{\kappa}\left(t-2,\frac{1}{2}; 3\right)}{\sum_{j=0}^{2}\Pi_{\kappa}\left(t-j,\frac{1}{2}; 3\right)} \quad \label{trianY} \\ && \qquad \qquad \qquad \qquad \qquad \qquad \qquad \qquad t\in \mathbb{R}. \nonumber
\end{eqnarray}
We see that the $\epsilon$-shape is a triangle whose vertices are the landmarks. 
 As $\kappa$ is increased, shapes with broader round edges appear and gradually shrink to the barycenter of the triangle as $\kappa$ is continuously increased to infinity. Since Eqs. (\ref{trianX}) and (\ref{trianY}) are invariant under the transformations $t \to t+3k$ $k\in \mathbb{Z}$ we find that any cyclic permutation of the landmarks yields the same $\kappa$-family of shapes, i.e. the choices of  landmarks
 \begin{equation}
\left(
\begin{array}{ccc}
7  & 0  & 4    \\
1  & 0  & 3
\end{array}
\right) \quad  \left(
\begin{array}{ccc}
 4  & 7 & 0   \\
 3  & 1 & 0
\end{array}
\right)
\end{equation}
correspond to the same $\kappa$-family. 
 
\begin{figure*}
\begin{center}
\includegraphics[width=0.9 \textwidth]{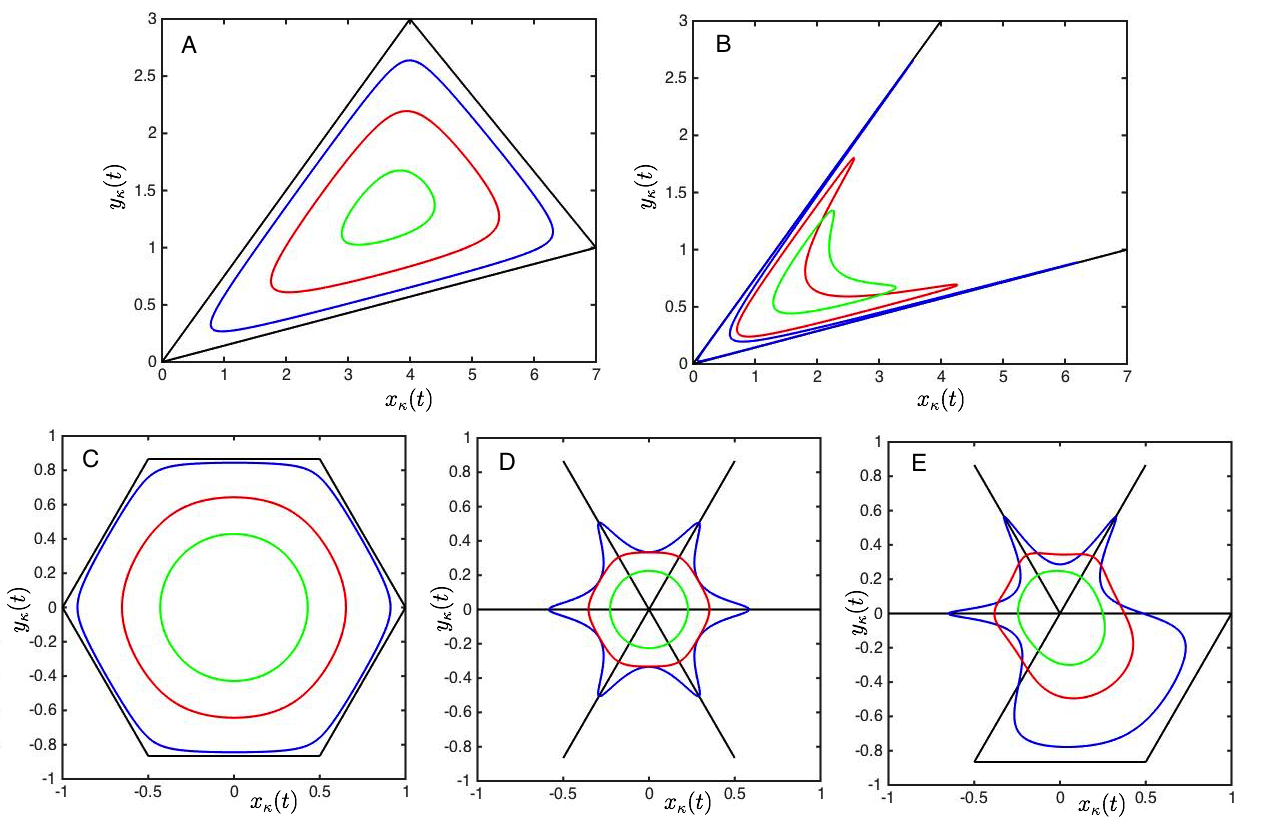}
\caption{\scriptsize{Plots of $y_{\kappa}(t)$ vs. $x_{\kappa}(t)$ obtained from Eqs. (\ref{2Dtype2x}) and (\ref{2Dtype2y})) for the landmarks given by (A) Eq. (\ref{A}), (B) Eq. (\ref{B}),(C) Eq. (\ref{C}) with $P=6$, (D) Eq. (\ref{D}) with $P=6$ and (E) Eq. (\ref{E}). The values of $\kappa$ are $\kappa=0.01$ (black curves), $\kappa=0.3$ (blue curves), $\kappa=0.5$ (red curves) and $\kappa=1$ (green curves). }} \label{polyg}
\end{center}
\end{figure*} 
 
If some of the landmarks are repeated, the resulting $\kappa$-family is different, even when parts of the outline can be the same. It is clear why: repeating a landmark makes it to acquire more weight therefore displacing the centroid. For example, if we consider the landmarks
\begin{equation}
\left(
\begin{array}{ccccc}
0  & 4  & 0 & 7 & 0   \\
0  & 3  & 0 & 1 & 0 \label{B}
\end{array}
\right)
\end{equation}
the landmark $(0,0)$ has thrice the weight of the other landmarks. Furthermore, the connection between landmarks $(4, 3)$ and $(7,1)$ dissappears.  From Eqs. (\ref{2Dtype2x}) and (\ref{2Dtype2y})), the mathematical model of the resulting $\kappa$-family is, 
\begin{eqnarray}
x_{\kappa}(t)&=&\frac{4\Pi_{\kappa}\left(t-1,\frac{1}{2}; 5\right)+7\Pi_{\kappa}\left(t-3,\frac{1}{2}; 5\right)}{\sum_{j=0}^{4}\Pi_{\kappa}\left(t-j,\frac{1}{2}; 5\right)}  \label{BtrianX} \\
y_{\kappa}(t)&=&\frac{3\Pi_{\kappa}\left(t-1,\frac{1}{2}; 5\right)+\Pi_{\kappa}\left(t-3,\frac{1}{2}; 5\right)}{\sum_{j=0}^{4}\Pi_{\kappa}\left(t-j,\frac{1}{2}; 5\right)}  \label{BtrianY}\\ && \qquad\qquad\qquad\qquad\qquad\qquad\qquad\qquad t\in \mathbb{R}.\nonumber
\end{eqnarray}
In Fig. \ref{polyg}B the curves obtained from Eqs.(\ref{BtrianX}) and (\ref{BtrianY}) are plotted for $\kappa=0.01$ (black curve), $\kappa=0.3$ (blue curve), $\kappa=0.5$ (red curve) and $\kappa=1$ (green curve). It is observed that the $\epsilon$-shape in this case is no longer a triangle because the connection between the points $(4,3)$ and $(7,1)$ is missing. Because of the higher weight of the origin, the shapes shrink to a point that is closer to the origin than the barycenter of the triangle in Fig \ref{polyg}A.
 
Any regular polygon defines a shape in which the vertices can be taken as landmarks of a $\kappa$-family of shapes. Therefore, a polygon of $P$ sides has landmarks
\begin{equation}
\left(
\begin{array}{ccccc}
\cos \left(\frac{2\pi}{P} \right)  & \cos \left(\frac{4\pi}{P} \right)  & \ldots & \cos \left(\frac{2\pi (P-1)}{P} \right) &\cos \left(\frac{2\pi P}{P} \right)    \\
\sin \left(\frac{2\pi}{P} \right)  & \sin \left(\frac{4\pi}{P} \right)  & \ldots & \sin \left(\frac{2\pi (P-1)}{P} \right) & \sin \left(\frac{2\pi P}{P} \right)
\end{array}
\right) \label{C}
\end{equation}
and Eqs. (\ref{2Dtype2x}) and (\ref{2Dtype2y}) become
\begin{eqnarray}
x_{\kappa}(t)&=&\frac{\sum_{j=0}^{P-1}\cos \left(\frac{2\pi (j+1)}{P} \right) \Pi_{\kappa}\left(t-j,\frac{1}{2}; P\right)}{\sum_{j=0}^{P-1}\Pi_{\kappa}\left(t-j,\frac{1}{2}; P\right)}, \label{polyx} \\
y_{\kappa}(t)&=&\frac{\sum_{j=0}^{P-1}\sin \left(\frac{2\pi (j+1)}{P} \right) \Pi_{\kappa}\left(t-j,\frac{1}{2}; P\right)}{\sum_{j=0}^{P-1}\Pi_{\kappa}\left(t-j,\frac{1}{2}; P\right)}\label{polyy} \\ &&\qquad \qquad\qquad \qquad\qquad \qquad\qquad \qquad  t\in \mathbb{R}  \nonumber
\end{eqnarray}
In Fig. \ref{polyg}C the curves obtained from these equations are plotted for $P=6$,  $\kappa=0.01$ (black curve), $\kappa=0.3$ (blue curve), $\kappa=0.5$ (red curve) and $\kappa=1$ (green curve). It is observed that as $\kappa$ is increased, the shapes approach circumferences. This is obvious: the regularity of the polygon obtained in the limit $\kappa \to 0$ (the $\epsilon$-shape), shrinks to the center of the polygon and, as $\kappa$ is increased the shape becomes uniformized because the details of the vertices are lost. We see that in increasing $\kappa$ more symmetric shapes are generally obtained since one passes from objects with discrete symmetries described by dihedral groups to the continuous symmetry of the circumference.

If we join each vertex of the polygon to the center we also construct shapes 
involving the same symmetry groups but with the appearance of stars. The landmarks are now
 \begin{equation}
\left(
\begin{array}{ccccccc}
0 & \cos \left(\frac{2\pi}{P} \right)  & 0 & \cos \left(\frac{4\pi}{P} \right)  & \ldots  &0 & \cos \left(\frac{2\pi P}{P} \right)    \\
0& \sin \left(\frac{2\pi}{P} \right)  & 0& \sin \left(\frac{4\pi}{P} \right)  & \ldots &  0& \sin \left(\frac{2\pi P}{P} \right)
\end{array} 
\right) \label{D}
\end{equation}
\begin{figure*}
\begin{center}
\includegraphics[width=0.7 \textwidth]{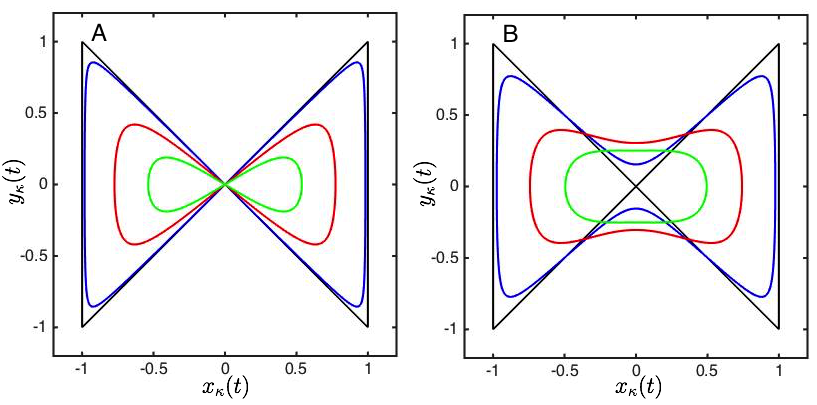}
\caption{\scriptsize{Plots of $y_{\kappa}(t)$ vs. $x_{\kappa}(t)$ obtained from Eqs. (\ref{2Dtype2x}) and (\ref{2Dtype2y}) for the landmarks given by (A) Eq. (\ref{lem1}) and (B) (\ref{lem2})). The values of $\kappa$ are $\kappa=0.01$ (black curves), $\kappa=0.3$ (blue curves), $\kappa=0.5$ (red curves) and $\kappa=1$ (green curves).}} \label{lems}
\end{center}
\end{figure*}

Eqs. (\ref{2Dtype2x}) and (\ref{2Dtype2y}) now become
\begin{eqnarray}
x_{\kappa}(t)&=&\frac{\sum_{j=0}^{P-1}\cos \left(\frac{2\pi (j+1)}{P} \right) \Pi_{\kappa}\left(t-2j-1,\frac{1}{2}; 2P\right)}{\sum_{j=0}^{2P-1}\Pi_{\kappa}\left(t-j,\frac{1}{2}; 2P\right)} \qquad \label{polyx} \\
y_{\kappa}(t)&=&\frac{\sum_{j=0}^{P-1}\sin \left(\frac{2\pi (j+1)}{P} \right) \Pi_{\kappa}\left(t-2j-1,\frac{1}{2}; 2P\right)}{\sum_{j=0}^{2P-1}\Pi_{\kappa}\left(t-j,\frac{1}{2}; 2P\right)} \qquad \label{polyy} \\&& \qquad \qquad \qquad \qquad\qquad \qquad\qquad \qquad t\in \mathbb{R} \nonumber
\end{eqnarray}
In Fig. \ref{polyg}D the curves obtained from these equations are plotted for $P=6$,  $\kappa=0.01$ (black curve), $\kappa=0.3$ (blue curve), $\kappa=0.5$ (red curve) and $\kappa=1$ (green curve). It is observed that as $\kappa$ is increased, the shapes approach again circumferences. However, for $\kappa$ sufficiently small, the shapes are not convex and their starry appearance resemble the one of certain species of \emph{Echinoderma}.

The shapes found within the $\kappa$-families and their deformations can be controlled by means of appropriately selected and weighted landmarks. Repeating a landmark after itself does not change the contour of the $\epsilon$-shape but has the effect, as $\kappa$ is increased, of drawing the shapes to a point nearer to the repeated landmark: the direction and amount of deformation within the shapes of a $\kappa$-family as $\kappa$ is increased can thus be \emph{absolutely} controlled. 

Any connected planar graph can be used as a structure of landmarks or $\epsilon$-shape. In Fig. \ref{polyg}E the same landmarks as in C and D are used but repeated and connected diferently, as the $\epsilon$-shape reveals. The $N=10$ landmarks are now
\begin{eqnarray}
&&\left(
\begin{array}{cccccccc}
0 & \cos \left(\frac{2\pi}{6} \right)  & 0 & \cos \left(\frac{4\pi}{6} \right)  & 0
& \cos \left(\frac{6\pi}{P} \right) & 0 & \cos \left(\frac{8\pi}{6} \right)  \\
0 & \sin \left(\frac{2\pi}{6} \right)  & 0 & \sin \left(\frac{4\pi}{6} \right)  & 0
& \sin \left(\frac{6\pi}{P} \right) & 0 & \sin \left(\frac{8\pi}{6} \right) 
\end{array}
\right.  \nonumber \\
&&\qquad \qquad \qquad \left.
\begin{array}{cc}
 \\
\cos \left(\frac{10\pi}{6} \right) & \cos \left(\frac{12\pi}{6} \right) \\
\sin \left(\frac{10\pi}{6} \right) & \sin \left(\frac{12\pi}{6} \right)
\end{array}
\right)
\label{E}
\end{eqnarray}
and Eqs. (\ref{2Dtype2x}) and (\ref{2Dtype2y}) yield shapes as those in Fig. \ref{polyg}E.

\begin{figure*}
\begin{center}
\includegraphics[width=0.7 \textwidth]{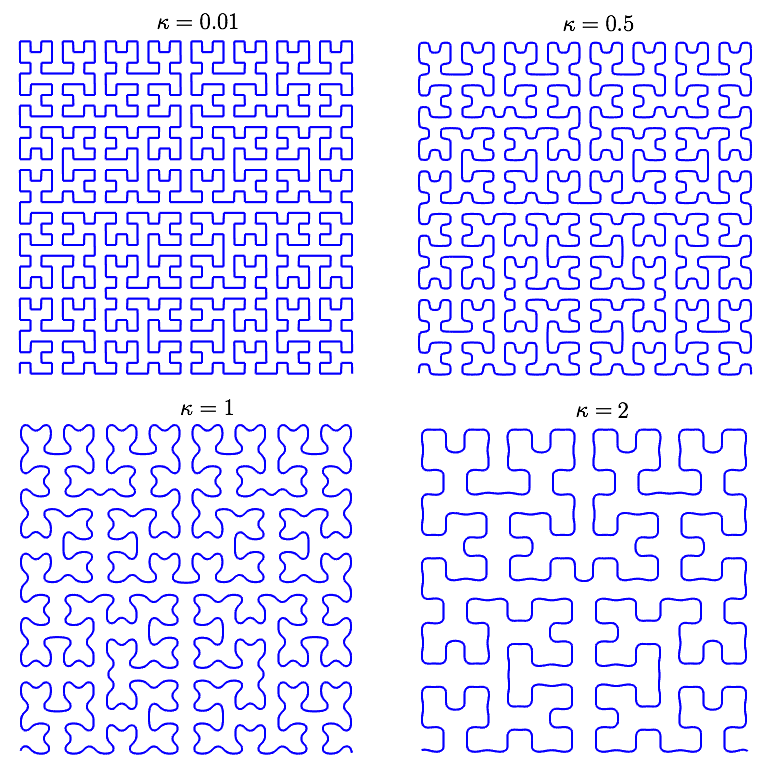}
\caption{\scriptsize{Plots of $y_{\kappa}(t)$ vs. $x_{\kappa}(t)$ within the unit square (axes not shown) obtained from Eqs. (\ref{2Dtype1x}) and (\ref{2Dtype1y}) for the landmarks obtained after $Q=5$ iterations of the transformation in Eq. (\ref{hil1}). The different values of $\kappa$ are shown on the panels.}} \label{hil1}
\end{center}
\end{figure*}

\begin{figure*}
\begin{center}
\includegraphics[width=0.7 \textwidth]{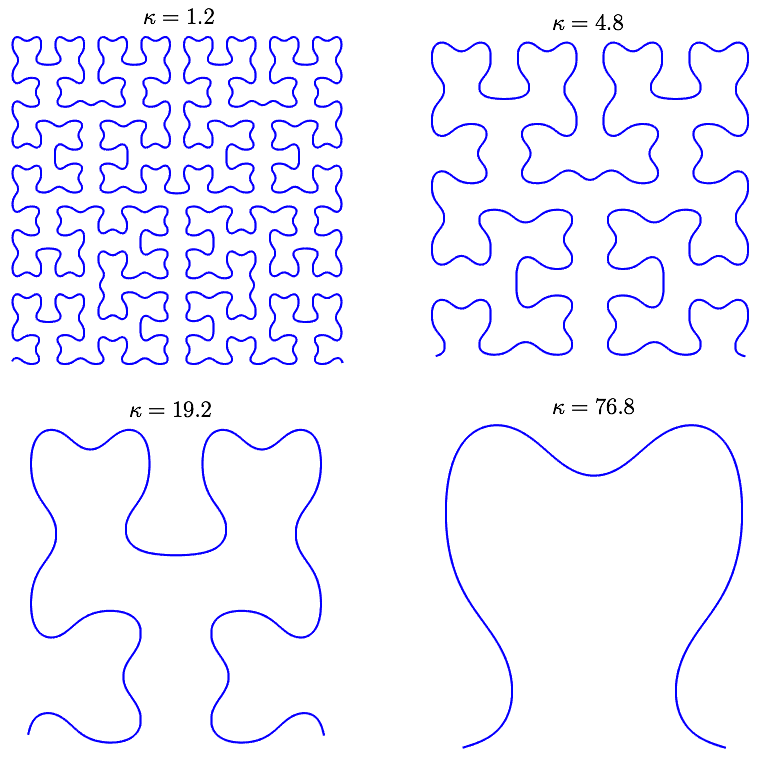}
\caption{\scriptsize{Plots of $y_{\kappa}(t)$ vs. $x_{\kappa}(t)$ within the unit square (axes not shown) obtained from Eqs. (\ref{2Dtype1x}) and (\ref{2Dtype1y}) for the landmarks obtained after $Q=5$ iterations of the transformation in Eq. (\ref{hil1}). The different values of $\kappa$ are shown on the panels. Each successive value of $\kappa$ is four times larger than the previous one from left to right and top to bottom.}} \label{hil2}
\end{center}
\end{figure*}

Let us assume that we do not know the parametric equations of a Bernoulli lemniscate and we find a plot of such a curve and want a mathematical model for it. Because of its symmetry we place its centroid at the origin. 
 We can then roughly select the vertices of a square as landmarks whose center of mass is the origin and such that the periodic $\epsilon$-shape crosses the origin two times through the diagonals of the square
\begin{equation}
\left(
\begin{array}{cccc}
-1  & 1  & 1 & -1    \\
-1  & 1  & -1 & 1 \label{lem1}
\end{array}
\right)
\end{equation}
These crossings occur in going from landmark $(-1, \ -1)$ to $(1, \ 1)$ and in going from  $(1, \ -1)$ to $(-1, \ 1)$. The landmarks given by Eq. (\ref{lem1}) can be replaced in Eqs. (\ref{2Dtype2x}) and (\ref{2Dtype2y})) yielding a mathematical model for the entire $\kappa$-family of shapes corresponding to these landmarks. In Fig. \ref{lems}A we plot $y_{\kappa}(t)$ vs. $x_{\kappa}(t)$ obtained from these equations for several $\kappa$ values. We observe that the $\kappa=1$ curve (green) reasonably approximates a lemniscate. 

Equal $\epsilon$-shapes may belong to different $\kappa$-families. If we now consider the landmarks
\begin{equation}
\left(
\begin{array}{cccccc}
-1  & 0 & 1  &  1 & 0 & -1    \\
-1  & 0 & -1  & 1 & 0 & 1 \label{lem2}
\end{array}
\right)
\end{equation}
the $\epsilon$-shape is identical as with the landmarks in Eq. (\ref{lem1}) but as $\kappa$ is varied, the shapes are different because they do not longer contain the origin: the latter is not a point where the shape crosses itself. This is shown in Fig. \ref{lems}B: the $\epsilon$-shape is the same as in Fig. \ref{lems}A but the $\kappa$-family is different, as it is revealed when $\kappa$ is increased. 

\subsection{Fractal shapes in 2D: Space-filling curves}

Landmarks can also be specified recursively by means of cellular automata \cite{VGM1, VGM2, VGM3, VGM7, Wolfram}  substitution and Lindenmayer systems \cite{VGM4,Wolfram,Prusinkiewicz} and digit replacement techniques \cite{CHAOSOLFRAC, VGM5,VGM6}. These procedures, when acting on real numbers, naturally lead to fractal structures \cite{CHAOSOLFRAC,VGM5,VGM6}.

A famous example of a fractal curve in the plane is Hilbert's curve \cite{Hilbert}. This curve, which constitutes an open shape, can be regarded as a substitution system. Starting with the landmark $X_0=0$, $Y_0=0$, we iteratively replace each previous landmark by four new landmarks following the transformation
\begin{equation}
\left(
\begin{array}{c}
X_{j}     \\
Y_{j} 
\end{array}
\right) \to 
\left(
\begin{array}{cccc}
\frac{Y_{j}-1}{2}  & \frac{X_{j}-1}{2} & \frac{X_{j}+1}{2} & \frac{1-Y_{j}}{2}    \\
\frac{X_{j}-1}{2}  & \frac{Y_{j}+1}{2}  & \frac{Y_{j}+1}{2}  & \frac{-1-X_{j}}{2} \label{hilbert}
\end{array}
\right)
\end{equation}
Let $Q$ denote the number of iterations. For $Q=0$ we have only as landmark the origin (initialization). For $Q=1$ we have the $N=4$ landmarks
\begin{equation}
\left(
\begin{array}{cccc}
-\frac{1}{2}  & -\frac{1}{2} & \frac{1}{2} & \frac{1}{2}    \\
-\frac{1}{2}  & \frac{1}{2}  & \frac{1}{2}  & -\frac{1}{2} \label{hilbertQ1}
\end{array}
\right)
\end{equation}
For $Q=2$ we have $4^2=16$ landmarks which are obtained replacing in order each of these landmarks following the prescription in Eq. (\ref{hilbert}). For $Q=n$ we obtain by this process, $4^{n}$ landmarks. Hilbert's curve is obtained by joining these landmarks by straight line segments.

By using Eqs. (\ref{2Dtype1x}) and (\ref{2Dtype1y}) we can now construct a $\kappa$-family of fractal curves having Hilbert's curve as $\epsilon$-shape. In Fig. \ref{hil1} we represent the curves calculated from Eqs. (\ref{2Dtype1x}) and (\ref{2Dtype1y}) for the 1024 landmarks obtained after $Q=5$ iterations and the values of $\kappa$ indicated on the panels. When $\kappa$ is vanishingly small, the $\epsilon$-shape (Hilbert's curve after $Q=5$ iterations) is approached. When $\kappa$ increases, the smallest details are averaged out and all straight edges and vertices become replaced by curved contours. As $\kappa$ increases further, 
we observe the shapes that we would have obtained with $Q=3$ iterations. Increasing $\kappa$ further then seems like `going backwards in the iterative process'. In the limit $\kappa \to \infty$ the shape collapses to its centroid, which corresponds to $Q=0$ iterations.

In Fig. \ref{hil2} we better illustrate this phenomenon. We observe that the curves are self-similar for any value of $\kappa$: By multiplying $\kappa$ by 4 we move backwards in the iteration process, obtaining the same shapes as if they were zoomed in, or as if they were obtained by the landmarks after $Q=5$, $Q=4$, $Q=3$ and $Q=2$ iterations (looking at the panels from left to right and top to bottom). We note that the $\mathcal{B}_{\kappa}$ function has a scaling property
\begin{equation}
\mathcal{B}_{a\kappa}(x,y)=\mathcal{B}_{\kappa}(ax,ay)
\end{equation}
This property is responsible for the phenomenon observed. We thus note that the smoothing parameter $\kappa$ is really a \emph{scale} parameter: as it is increased it captures the features of a shape in a more broader scale. Since the curves are self-similar we observe the same kind of structures as $\kappa$ is increased. 

\subsection{3D shapes}

In three dimensions we have the vector $\mathbf{r}_{\kappa}(t)=(x_{\kappa}(t), y_{\kappa}(t), z_{\kappa}(t))$ that yields a plane curve in parametric form and Eq. (\ref{type1}) for open shapes reduces to
\begin{eqnarray}
x_{\kappa}(t)&=&\sum_{j=0}^{N-1}X_{j}\frac{\mathcal{B}_{\kappa}\left(t-j,\frac{1}{2}\right)}{\mathcal{B}_{\kappa}\left(t-\frac{N-1}{2},\frac{N}{2}\right)} \label{3Dtype1x} \\
y_{\kappa}(t)&=&\sum_{j=0}^{N-1}Y_{j}\frac{\mathcal{B}_{\kappa}\left(t-j,\frac{1}{2}\right)}{\mathcal{B}_{\kappa}\left(t-\frac{N-1}{2},\frac{N}{2}\right)}  \label{3Dtype1y} \\ 
z_{\kappa}(t)&=&\sum_{j=0}^{N-1}Z_{j}\frac{\mathcal{B}_{\kappa}\left(t-j,\frac{1}{2}\right)}{\mathcal{B}_{\kappa}\left(t-\frac{N-1}{2},\frac{N}{2}\right)}  \label{3Dtype1z} \\ 
&& \qquad \qquad \qquad \qquad t\in [0,N-1], t\in \mathbb{R} \nonumber
\end{eqnarray}
where $(X_j, Y_j, Z_j)=\mathbf{r}_{j}$ specifies the $j$-th landmark. 

\begin{figure*}
\begin{center}
\includegraphics[width=0.8 \textwidth]{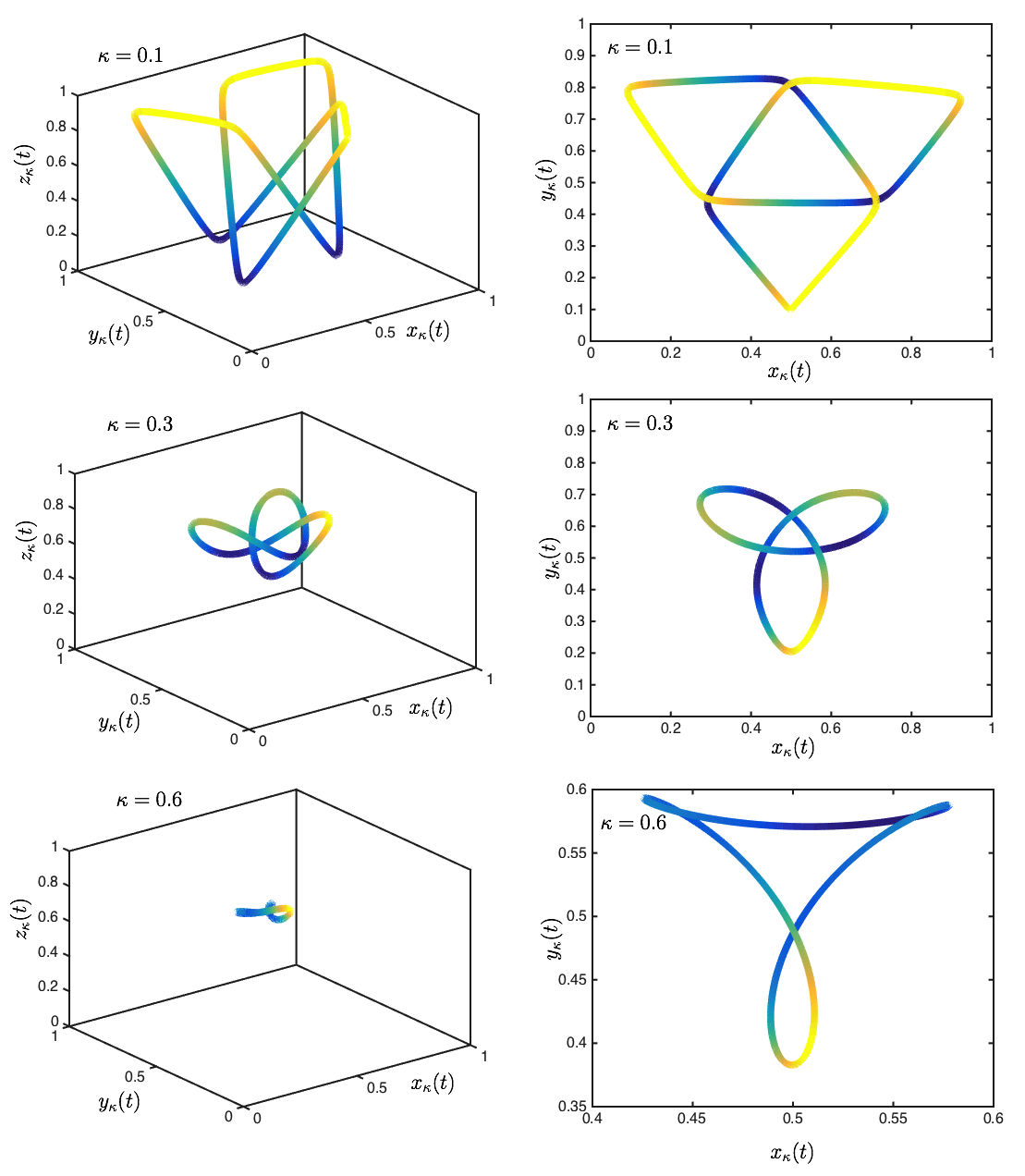}
\caption{Plots of $z_{\kappa}(t)$ vs. $y_{\kappa}(t)$ and $x_{\kappa}(t)$ of the curves obtained from Eqs. (\ref{3Dtype2x}) to (\ref{3Dtype2z}) for the landmarks given in Eq. (\ref{knolan}) and the values of $\kappa$ indicated on the figures. In the right column, the shapes are also projected on the $x_{\kappa}-y_{\kappa}$ plane for better visualization. The color code is related to the height $z_{\kappa}$ and ranges from dark blue ($z_{\kappa}=0$) to yellow ($z_{\kappa}=1$).} \label{knots}
\end{center}
\end{figure*}

For closed shapes in the plane Eq. (\ref{type2}) is equivalent to
\begin{eqnarray}
x_{\kappa}(t)&=&\frac{\sum_{j=0}^{N-1}X_{j}\Pi_{\kappa}\left(t-j,\frac{1}{2}; N\right)}{\sum_{j=0}^{N-1}\Pi_{\kappa}\left(t-j,\frac{1}{2}; N\right)} \label{3Dtype2x} \\ 
y_{\kappa}(t)&=&\frac{\sum_{j=0}^{N-1}Y_{j}\Pi_{\kappa}\left(t-j,\frac{1}{2}; N\right)}{\sum_{j=0}^{N-1}\Pi_{\kappa}\left(t-j,\frac{1}{2}; N\right)}  \label{3Dtype2y}
\\ 
z_{\kappa}(t)&=&\frac{\sum_{j=0}^{N-1}Z_{j}\Pi_{\kappa}\left(t-j,\frac{1}{2}; N\right)}{\sum_{j=0}^{N-1}\Pi_{\kappa}\left(t-j,\frac{1}{2}; N\right)}   \qquad t\in \mathbb{R} \quad \label{3Dtype2z}
\end{eqnarray}
The landmarks $(X_0, Y_0, Z_0)$, $(X_1, Y_1,Z_1)$, $\ldots$, $(X_{N-1}, Y_{N-1}, Z_{N-1})$ can be written more compactly in a rectangular matrix as 
\begin{equation}
\left(
\begin{array}{cccc}
X_0  & X_1  & \ldots & X_{N-1}    \\
Y_0  & Y_1  & \ldots & Y_{N-1}    \\
Z_0  & Z_1  & \ldots & Z_{N-1}
\end{array}
\right)
\end{equation}

As an example of 3D curve, we provide a mathematical model for a trefoil knot, a closed curve which features prominently in the Lorenz attractor \cite{Kaufman}. The landmarks of any knot can be found from its planar representation by selecting points on its contour. If a point lags at the bottom (resp. at the front) in the planar representation, a value of $Z=0$ (resp. $Z=1$) is adjoined to the coordinates $X$ and $Y$ read from the planar representation. We can pick, e.g. $N=9$ landmarks in this way, outlining the periodic contour of the shape
\begin{equation}
\left(
\begin{array}{cccccccccc}
0.5 & 0.3  & 0.5 & 1 & 0.7 & 0.3 & 0 & 0.5 & 0.7   \\
0 & 0.4 & 0.8 & 0.8 & 0.4 & 0.4 & 0.8 & 0.8 & 0.4    \\
1 & 0 & 1 & 1 & 0 & 1 & 1 & 0 & 1  \label{knolan}
\end{array}
\right).
\end{equation}
Eqs. (\ref{3Dtype2x}) to (\ref{3Dtype2z}) automatically provide a model for the $\kappa$-family of closed shapes derived from these landmarks. In Fig. \ref{knots} we plot several shapes that are members of this family. The shape of the trefoil knot shrinks towards its centroid. At $\kappa > 0.55$ aprox. the knot undoes and, therefore a knot/unknot transition is observed. Complemented with structural, biological information, the mathematical methods here presented may thus be of interest in the modelling of knot/unknot transitions found in macromolecules and DNA \cite{Wasserman, Liu}.

\subsection{$\kappa$-family of shapes from Lorenz' attractor} \label{Lorenz}

The celebrated Lorenz system
\begin{eqnarray}
\dot{x}&=&\sigma (y-x) \label{lor1} \\
\dot{y}&=&x(\rho -z)-y \label{lor2} \\
\dot{z}&=&xy-\beta z \label{lor3} 
\end{eqnarray}
can be embedded in a $\kappa$-family of shapes as described in Sec. \ref{dynapp}. The landmarks $\mathbf{r}_{j}:=(X_j, Y_j, Z_j)$ can be iteratively obtained from Eq. (\ref{evolsmooth}) as
\begin{eqnarray}
X_j&=&X_{j-1}+\delta \sigma \left(Y_{j-1}-X_{j-1}\right) \label{lor1D} \\
Y_j&=&Y_{j-1}+\delta \left[X_{j-1}\left(\rho -Z_{j-1}\right)-Y_{j-1}\right] \label{lor2D} \\
Z_j&=&Z_{j-1}+\delta \left[X_{j-1}Y_{j-1}-\beta Z_{j-1} \right]        \label{lor3D} 
\end{eqnarray}
starting from an initial condition $(X_0, Y_0, Z_0)$. In Fig. \ref{loratrac} $z_{\kappa}(t)$ is shown together with $y_{\kappa}(t)$ and $x_{\kappa}(t)$ obtained from Eqs. (\ref{3Dtype1x}) to (\ref{3Dtype1z}) and the landmarks given by Eqs. (\ref{lor1D}) to (\ref{lor3D}) for the initial condition $(X_0, Y_0, Z_0)=(2.5704,\ 3.6945,\ 16.4286)$, $N=1000$, $\delta=0.01$, $\rho=28$, $\sigma=10$, $\beta=8/3$ and the values of $\kappa$ indicated over the panels. It is observed that as $\kappa$ is decreased to 0, the $\epsilon$-shape approaches that of the Lorenz attractor. When $\kappa$ is increased, deformed complex attractors are obtained. The latter shrink to a point as $\kappa \to \infty$. The resulting, deformed attractors, for $\kappa >0$ finite are always found within the convex hull of the landmarks obtained from the Lorenz system.
 
\begin{figure*}
\begin{center}
\includegraphics[width=0.9 \textwidth]{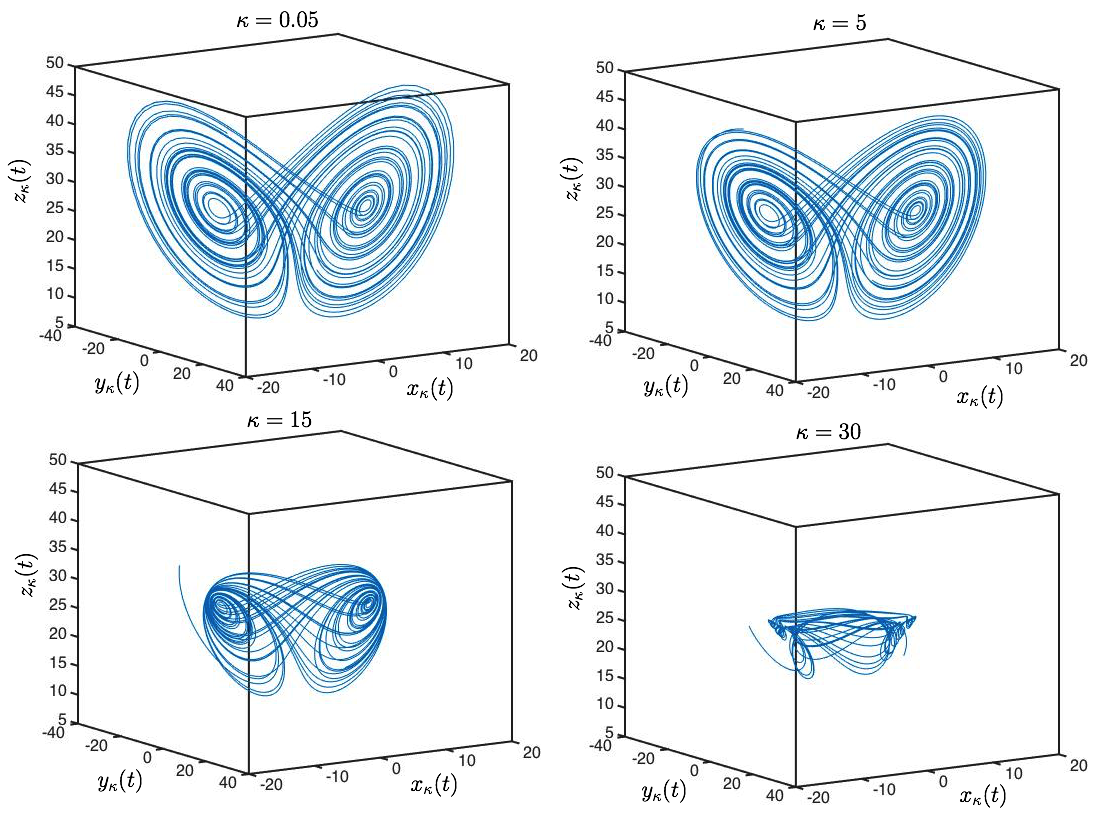}
\caption{Plots of $z_{\kappa}(t)$ vs. $y_{\kappa}(t)$ and $x_{\kappa}(t)$ obtained from Eqs. (\ref{3Dtype1x}) to (\ref{3Dtype1z}) and the landmarks given by Eqs. (\ref{lor1D}) to (\ref{lor3D}) for the initial condition $(X_0, Y_0, Z_0)=(2.5704,\ 3.6945,\ 16.4286)$, $N=1000$, $\delta=0.01$, $\rho=28$, $\sigma=10$, $\beta=8/3$ and the values of $\kappa$ indicated over the panels.} \label{loratrac}
\end{center}
\end{figure*}

\section{Conclusions} \label{conclusions}

In this article, a general theory of shape has been presented and illustrated with examples. The theory provides mathematical models for any shape that can be found in the physical world and requires the specification of a finite set of landmarks that fully characterize entire families of qualitatively related shapes. The theory addresses all shapes that are path connected (all shapes of interest in physics, engineering and the life sciences) although it may be extended (by introducing additional parameters) to other topologically interesting curves that are not path connected (as e.g. the Warsaw circle). Disconnected shapes that are composed of path connected pieces can be treated within the theory as \emph{scenes} (compound shapes).

The theory of shape developed in this article has been linked to viability theory \cite{Aubin3, Aubin1, Aubin2} and invariant sets of dynamical systems \cite{Blanchini1, Blanchini2}. When a shape is looked upon as a trajectory of a complex dynamical system (possibly under the influence of unpredictable perturbations) infinite families of viable trajectories can be constructed for the dynamical system, from the knowledge of a finite number of points (landmarks) on a measured trajectory.  All viable trajectories generated as $\kappa$-families of shapes are found in any convex set of the landmarks and, specifically, in the convex hull of the latter (the convex set of minimal size).

The mathematical models for the shapes correspond to nonlinear $\mathcal{B}_{\kappa}$-embeddings \cite{JPHYSCOMPLEX}. The latter structures are nonlinear functions of the time $t$ and deformation $\kappa$ parameters, but are linear functions of the landmarks. This makes straightforward the application of transformations to the shapes (e.g. translations, rotations, rescalings).
Specifically, procrustean analysis can easily be carried on the $\kappa$-families of shapes, after displacing the centroid to the origin and rescaling the shapes appropriately \cite{Kendall1, Kendall2}. The former operation is carried out subtracting the coordinates of the centroid to each of the landmarks. The second operation proceeds by multiplying all the landmarks by a suitable scalar to carry the shape to the confines of the unit interval, square or box. It is not necessary to rotate the shapes within the same $\kappa$-family to compare them because all of them share the same orientation. 

The theory explains how shapes can be hierarchically integrated in increasingly complex structures (scenes) so that shape is \emph{transferred} to more encompassing entities. The time unfolding that characterizes the shape is entirely transferred. The theory also explains how structural information can be retrieved at any level (the $\epsilon$-shapes or $\epsilon$-scenes, that contain this information, are included as particular cases within the $\kappa$-families derived from them). In this sense the theory satisfies Leyton's criteria for a succesful theory of shape \cite{Leyton}. Furthermore, within a same $\kappa$-family, shapes with larger $\kappa$ values are generally more symmetrical that shapes with lower $\kappa$, the $\epsilon$-shape being the least symmetrical shape within a $\kappa$-family. This allows to describe within our theory Leyton's processes in which symmetry-breaking serves as memory storage. These processes are found in Leyton's application of its theory of shape to painting \cite{Leypain} and architecture \cite{Leyarch}. Although a previous version of Leyton's theory \cite{Leymind} has been strongly criticized on mathematical grounds and vagueness of the ideas \cite{Wagemans} that critique did not incorporate the wreath product construction in \cite{Leyton}, which in, our view, constitutes a sound mathematical basis for Leyton's theory of shape. Our theory of shape presented in this manuscript yields independent support to Leyton's main basic idea that perception is nothing but the recovery of causal history \cite{Leyton}. Indeed if the parameter $\kappa$ in our theory is viewed as a time variable, all shapes within a $\kappa$-family are causally related, as they are perceptually, to the $\epsilon$-shape. 

We can mention some interesting directions for further research. First of all, we have described $\kappa$-families of shapes in Euclidean spaces but we can also be interested in constraining the shapes so as to lie on a manifold \cite{Jupp, Kim}. Another interesting direction is investigating the relationship between our theory and Leyton's wreath product in \cite{Leyton} and the relationships of $\kappa$-families with group actions, in general. Finally, a very challenging but interesting problem, is to investigate the connection of $\kappa$-families to continuous walks through Kendall shape spaces \cite{Kendall1, Kendall2, Klingenberg}.
~\\
~\\

\section*{Acknowledgments}

We have benefitted from fruitful conversations with Prof. Jos\'e A. Manzanares. Financial support under project No. PGC2018-097359-B-I00 from \emph{Ministerio de Ciencia, Innovaci\'on y Universidades} (Spain) and the European Regional Development Funds (FEDER) is also gratefully acknowledged.

\end{document}